\documentclass{amsart}

\usepackage{amsthm,amsmath,amssymb,amsfonts,epsfig}
\usepackage{hyperref}
\usepackage{stmaryrd}
\usepackage{epstopdf}

\numberwithin{equation}{section}

\theoremstyle{plain}
\newtheorem{theorem}{Theorem}[section]

\newtheorem{lemma}[theorem]{Lemma}

\newtheorem{corollary}[theorem]{Corollary}
\theoremstyle{definition}

\theoremstyle{remark}
\newtheorem{remark}[theorem]{Remark}
\newtheorem{example}[theorem]{Example}
\newcounter{counter_a}

\newcounter{counter_b}

\newcommand{\cL}{\mathcal{L}}

\newcommand{\Dom}{{\rm Dom}}

\newcommand{\Spec}{{\rm Spec}}
\newcommand{\Span}{{\rm span}}
\newcommand{\Ker}{{\rm Ker}}

\renewcommand{\Re}{{\rm Re}\;}
\renewcommand{\Im}{{\rm Im}\;}

\newcommand{\efrac}[2]{\genfrac{}{}{0ex}{}{#1}{#2}}

\newcommand{\dist}{{\rm dist}}

\DeclareMathOperator\spn{span}
\newcommand{\ess}{\mathrm{ess}}
\newcommand{\rank}{\mathrm{Rank}}
\newcommand{\range}{\mathrm{Range}}
\newcommand{\dis}{\mathrm{dis}}

\newcommand\void[1]{}
\begin{document}
\title[The Galerkin Method for Perturbed Self-Adjoint Operators]{The Galerkin Method for Perturbed Self-Adjoint Operators and Applications}
\author[M. Strauss]{Michael Strauss$^{1}$}
\begin{abstract}
We consider the Galerkin method for approximating the spectrum of
an operator $T+A$ where $T$ is semi-bounded self-adjoint and
$A$ satisfies a relative compactness condition. We show that the method is reliable
in all regions where it is reliable for the unperturbed problem - which always contains $\mathbb{C}\backslash\mathbb{R}$. The results
lead to a new technique for identifying eigenvalues of $T$, and for identifying spectral pollution which arises from applying the Galerkin
method directly to $T$. The new technique benefits from being applicable on the form domain.
$^1$WIMCS-Leverhulme Fellow, School of Mathematics, Cardiff University,
Senghennydd Road, CARDIFF CF24 4AG, Wales, UK. straussmd@cardiff.ac.uk.
\end{abstract}

\date{\today}
\maketitle

\section{Introduction}

In general, the approximation of the spectrum of a semi-bounded self-adjoint a operator $T$ with the Galerkin (finite section) method is reliable only for those eigenvalues lying below the essential spectrum. Any element from the closed convex hull of  $\sigma_{\ess}(T)\cup\{\infty\}$ can in principle be the limit point of a sequence
of Galerkin eigenvalues; see \cite[Theorem 2.1]{lesh}. This phenomenon is called spectral pollution and constitutes a serious problem
in computational spectral theory; see for example \cite{boff,boff2,dasu,rapp}. As is well-known, under fairly mild assumptions on a sequence of trial spaces the
Galerkin method will capture the whole spectrum, although eigenvalues may be obscured by spectral pollution. In the presence
of essential spectrum, the Galerkin method applied to non-self-adjoint operators is less well understood. Reliability of the Galerkin method is assured in some situations, notably for compact operators or operators with compact resolvent and satisfying ellipticity conditions; see for example \cite{chat} and reference therein.

The closed sesquilinear form associated to $T$ we denote by $\frak{t}$. Let $A$ be a closed, $\vert T\vert^{\frac{1}{2}}$-compact operator,
such that $(T+A)^* = T+A^*$ and for some $0\le\alpha <1$, $\beta\ge 0$ satisfies
\begin{equation}\label{condition}
\vert\langle A\phi,\phi\rangle\vert\le\alpha\frak{t}[\phi] + \beta\Vert\phi\Vert^2\quad\textrm{for all}\quad\phi\in\Dom(\frak{t})=\Dom(\vert T\vert^{\frac{1}{2}}).
\end{equation}
We note that $\sigma_{\ess}(T+A)=\sigma_{\ess}(T)$; see for example \cite[Theorem IV.5.35]{katopert}.
In Section 2 we shall be concerned with the
approximation of the discrete spectrum (isolated eigenvalues of finite multiplicity) of $T+A$
by means of the Galerkin method. In Theorem \ref{resolvent_bound} we show that if a compact set $\Gamma\subset\rho(T+A)$ does not contain a
limit point of Galerkin eigenvalues for the self-adjoint operator $T$ and sequence of trial spaces $(\cL_n)$,
then $\Gamma$ will not contain a limit point of Galerkin eigenvalues for the operator $T+A$. Further, if in a region $U\subset\mathbb{C}$ with
$U\cap\sigma_{\ess}(T)=\varnothing$, the self-adjoint operator $T$ does not suffer from spectral pollution for a sequence of trial spaces $(\cL_n)$,
then Theorem \ref{spectralinc} states that $\sigma(T+A)\cap U$ will be approximated and without spectral pollution, and by Theorem \ref{multthm}
the Galerkin method will also capture the multiplicity of eigenvalues in the region.

In Section 3 we employ the preceding results to aid the development of a new technique for approximating those eigenvalues of $T$ which lie within the closed convex hull of $\sigma_{\ess}(T)\cup\{\infty\}$ - the region where a direct application of the Galerkin method is unreliable.
This is an issue which has received considerable attention in recent decades. There are two approaches to the problem. Firstly, general methods which
may be applied to an arbitrary self-adjoint operator, and secondly, methods designed for a specific class of operator. Examples of the former are proposed in
\cite{DP,lesh,shar,zim} and can be highly effective. However, application of these techniques can require \emph{a priori} information about the spectrum, and always require trial subspaces to belong to the operator domain rather than the preferred form domain. The latter is due to requiring matrices with entries of the form $\langle T\phi,T\psi\rangle$, while the Galerkin method requires only matrices with entries of the form $\frak{t}(\phi,\psi)$. A new method designed for a specific class of operator is studied in \cite{mar1,mar2} and is applicable to operators of the form $T=-\Delta + q$ which act on $L^2(\mathbb{R}^n)$ and $L^2(0,\infty)$, respectively. The idea is to apply the Galerkin method to
the perturbed operator $-\Delta + q + is$ for a suitably chosen function $s$. The result of the perturbation is to lift eigenvalues of $T$ off the
real line where they can be approximated without encountering spectral pollution. Based on this idea and using the form domain, we consider the Galerkin method applied to $T+iQ$
where $Q$ is an orthogonal projection. For a set $J\subset\mathbb{R}$ and a trial space $\cL$ we choose $Q$ to be the orthogonal projection onto the eigenspace
associated to Galerkin eigenvalues contained in $J$, we then apply the Galerkin method to $T+iQ$ and a \emph{larger} trial space. In Theorem \ref{QQ} we show that
if $\lambda\in\sigma_{\dis}(T)\cap J$ and our trial space $\cL$ approximates the corresponding eigenspace sufficiently well, then $T+iQ$ will have eigenvalues in
a neighbourhood of $\lambda+i$ with total-multiplicity equally that of $\lambda\in\sigma_{\dis}(T)$. Now applying results from Section 2, we may approximate the
non-real eigenvalues of $T+iQ$, their multiplicities, and free from spectral pollution. The technique is effectively applied to examples where eigenvalues are
obscured by spectral pollution.

\section{The Truncated Eigenvalue Problem}
We set $m=\min\sigma(T)$ and denote by $\mathcal{H}_{\frak{t}}$ the Hilbert space $\Dom(\frak{t})$ equipped
with the inner product
\begin{displaymath}
\langle\phi,\psi\rangle_{\frak{t}} := (\frak{t}-m)(\phi,\psi) + \langle\phi,\psi\rangle\quad\textrm{and norm}\quad
\Vert\phi\Vert_{\frak{t}}^2 := (\frak{t}-m)[\phi] + \Vert\phi\Vert^2.
\end{displaymath}
Throughout, $(\cL_n)\subset\Dom(\frak{t})$ is a sequence of finite-dimensional subspaces.
The orthogonal projections from $\mathcal{H}$ and $\mathcal{H}_{\frak{t}}$ onto $\cL_n$ will be denoted by $P_n$ and $\hat{P}_n$, respectively. We always assume that the sequence $(\cL_n)$ is dense in $\mathcal{H}_{\frak{t}}$:
\begin{displaymath}
\forall\phi\in\Dom(\frak{t})\quad\exists\phi_n\in\cL_n:\quad\Vert\phi-\phi_n\Vert_{\frak{t}}\to0.
\end{displaymath}
We define the form
\begin{displaymath}
\frak{s}(\phi,\psi) := \frak{t}(\phi,\psi) + \langle A\phi,\psi\rangle\quad\textrm{with}\quad\Dom(\frak{s})=\Dom(\frak{t}),
\end{displaymath}
the sets
\begin{align*}
\sigma(T+A,\cL_n)&:=\{z\in\mathbb{C}:~\exists\phi\in\cL_n\quad\textrm{with}\quad\frak{s}(\phi,\psi) = z\langle\phi,\psi\rangle~\forall\psi\in\cL_n\}\label{fs}\\
\sigma(T,\cL_n)&:=\{z\in\mathbb{C}:~\exists\phi\in\cL_n\quad\textrm{with}\quad\frak{t}(\phi,\psi) = z\langle\phi,\psi\rangle~\forall\psi\in\cL_n\},
\end{align*}
and the limit sets
\begin{align*}
\sigma(T+A,\cL_\infty)&:= \{z\in\mathbb{C}:~\exists~z_{n}\in\sigma(T+A,\cL_{n})\textrm{ with }z_{n}\to z\}\\
\sigma(T,\cL_\infty)&:= \{z\in\mathbb{C}:~\exists~z_{n}\in\sigma(T,\cL_{n})\textrm{ with }z_{n}\to z\}.
\end{align*}
Associated to the restriction of $\frak{s}$ and $\frak{t}$ to $\cL_n$ are operators
$S_n$ and $T_n$ which act on the Hilbert space $\cL_n$,
and satisfy
\[
\langle S_n\phi,\psi\rangle = \frak{s}(\phi,\psi)\quad\textrm{and}\quad\langle T_n\phi,\psi\rangle = \frak{t}(\phi,\psi).
\]
Evidently, we have $\sigma(T+A,\cL_n) = \sigma(S_n)$ and $\sigma(T,\cL_n)=\sigma(T_n)$.

\void{\subsection{Matrix Representations}

Let the vectors $\{\phi_1,\dots,\phi_m\}$ form a basis for the subspace $\cL_n$ for some fixed $n\in\mathbb{N}$. Consider the
matrices
\begin{displaymath}
[\mathcal{S}_n]_{i,j} = \frak{s}(\phi_j,\phi_i),\quad[\mathcal{T}_n]_{i,j} = \frak{t}(\phi_j,\phi_i),\quad\textrm{and}\quad[\mathcal{M}_n]_{i,j} = \langle\phi_j,\phi_i\rangle.
\end{displaymath}
The solutions to \eqref{fs} are precisely the solutions to the matrix eigenvalue problems
\begin{displaymath}
\mathcal{S}_n\underline{u} =  z\mathcal{M}_n\underline{u}\quad\textrm{and}\quad \mathcal{M}_n^{-\frac{1}{2}}\mathcal{S}_n\mathcal{M}_n^{-\frac{1}{2}}\underline{u} =  z\underline{u}.
\end{displaymath}}

\subsection{Regular Sets}

We say that $\Gamma\subset\mathbb{C}$ is a $T_n$-regular set if there exist a $\delta>0$ and $N\in\mathbb{N}$, with
\begin{equation}\label{reg}
\min_{\efrac{\phi\in\cL_n}{\Vert\phi\Vert=1}}\max_{\efrac{\psi\in\cL_n}{\Vert\psi\Vert=1}}
\vert(\frak{t}-z)(\phi,\psi)\vert
\ge\delta\quad\textrm{for all}\quad z\in\Gamma,\quad\textrm{and}\quad n\ge N,
\end{equation}
or equivalently
\begin{equation}\label{reg1}
\Vert(T_n  - z)\phi\Vert\ge\delta\Vert\phi\Vert\quad\textrm{for all}\quad z\in\Gamma,\quad\phi\in\mathcal{L}_n\quad\textrm{and}\quad n\ge N.
\end{equation}
Similarly, we define $S_n$-regular sets, and we shall make use of the function
\begin{align*}
\sigma_n(z)&:=\min_{\efrac{\phi\in\cL_n}{\Vert\phi\Vert=1}}\max_{\efrac{\psi\in\cL_n}{\Vert\psi\Vert=1}}
\vert(\frak{s}-z)(\phi,\psi)\vert\\
&=\min_{\efrac{\phi\in\cL_n}{\Vert\phi\Vert=1}}\Vert(S_n - z)\phi\Vert\\
&= \begin{cases}
  \Vert(S_n - z)^{-1}\Vert^{-1} & \text{if } z\in\rho(S_n), \\[1ex]
  0 &  \text{if } z\in\sigma(S_n).
  \end{cases}
\end{align*}

\begin{lemma}\label{resreg}
If $\Gamma$ is a $T_n$-regular set, then $\Gamma\subset\rho(T)$.
\end{lemma}
\begin{proof}
We suppose the contrary, so that $\Gamma$ is a $T_n$-regular set and $\lambda\in\Gamma\cap\sigma(T)$. There exists a normalised sequence $(\psi_k)\subset\Dom(T)$ such that $\Vert(T-\lambda)\psi_k\Vert< k^{-1}$. For a fixed $k$, let $\hat{\psi}_n=\hat{P}_n\psi_k$. Then for any normalised $v_n\in\cL_n$ we have
\begin{align*}
\vert\frak{t}(\hat{\psi}_n,v_n) - \lambda\langle\hat{\psi}_n,v_n\rangle\vert
&=\vert\frak{t}(\hat{\psi}_n-\psi_k,v_n) + \frak{t}(\psi_k,v_n) - \lambda\langle\hat{\psi}_n,v_n\rangle\vert\\
&=\vert\langle\hat{\psi}_n-\psi_k,v_n\rangle_{\frak{t}} + (m-1)\langle\hat{\psi}_n-\psi_k,v_n\rangle\\
&\qquad+ \langle(T - \lambda)\psi_k,v_n\rangle- \lambda\langle\hat{\psi}_n-\psi_k,v_n\rangle\vert\\
&=\vert\langle(\hat{P}_n-I)\psi_k,v_n\rangle_{\frak{t}}+ \langle(T - \lambda)\psi_k,v_n\rangle\\
&\qquad- (\lambda-m+1)\langle\hat{\psi}_n-\psi_k,v_n\rangle\vert\\
&=\vert\langle(T - \lambda)\psi_k,v_n\rangle- (\lambda-m+1)\langle\hat{\psi}_n-\psi_k,v_n\rangle\vert\\
&< k^{-1} + \vert\lambda-m+1\vert\Vert\hat{\psi}_n-\psi_k\Vert,
\end{align*}
where the right hand side is less than $k^{-1}$ for all sufficiently large $n$. Since $\Vert\hat{\psi}_n\Vert\to 1$ it follows that $\Gamma$ is not a $T_n$-regular set. The result follows from the contradiction.
\end{proof}

\begin{lemma}\label{Alemma}
Let $(\phi_n)$ be a bounded sequence of vectors with $\phi_n\in\cL_n$ and
\[\max_{\efrac{v\in\cL_n}{\Vert v\Vert=1}}\vert(\frak{s}-z)(\phi_n,v) - \langle x,v\rangle\vert\to 0\quad\textrm{as}\quad n\to\infty.\]
There exists a $u\in\mathcal{H}$ and a subsequence $n_k$, such that
$\Vert A\phi_{n_k} - u\Vert\to 0$. Moreover, if $\{z\}$ is a $T_n$-regular set, then
$\Vert\phi_{n_k}-(T-z)^{-1}(x-u)\Vert_{\frak{t}}\to 0$.
\end{lemma}
\begin{proof}
Suppose that $A\phi_n$ does not have a convergent subsequence. It follows that $\Vert\vert T\vert^{\frac{1}{2}}\phi_n\Vert\to\infty$ and therefore also that $\frak{t}[\phi_n]\to\infty$. We have
\begin{align*}
\langle x,\phi_n\rangle\approx(\frak{s}-z)(\phi_n,\phi_n) = \frak{t}[\phi_n] + \langle A\phi_n,\phi_n\rangle -z\langle\phi_n,\phi_n\rangle.
\end{align*}
Let $M\in\mathbb{R}$ be such that $\Vert\phi_n\Vert\le M$ for all $n\in\mathbb{N}$. Using \eqref{condition} and recalling that $m=\min\sigma(T)$, we obtain
\begin{align*}
\frak{t}[\phi_n] &\approx z\langle\phi_n,\phi_n\rangle - \langle A\phi_n,\phi_n\rangle + \langle x,\phi_n\rangle\\
&\le\vert z\vert M^2 + \vert \langle A\phi_n,\phi_n\rangle\vert + \Vert x\Vert M\\
&\le\vert z\vert M^2+\alpha\langle\vert T\vert\phi_n,\phi_n\rangle + \beta M^2 + \Vert x\Vert M\\
&\le\vert z\vert M^2+\alpha\frak{t}[\phi_n] + 2\alpha\vert m\vert M^2+ \beta M^2 + \Vert x\Vert M.
\end{align*}
Therefore
\begin{align*}
\frak{t}[\phi_n] \le\frac{\vert z\vert M^2+ 2\alpha\vert m\vert M^2+ \beta M^2+ \Vert x\Vert M}{1-\alpha}
\end{align*}
which is a contradiction since the left hand side converges to $\infty$. We deduce that $A\phi_{n_k}\to u$ for some $u\in\mathcal{H}$
and subsequence $n_k$.

Suppose now that $\{z\}$ is a $T_n$-regular set. Then by Lemma \ref{resreg} we have $z\in\rho(T)$, hence there exists a vector $\psi\in\Dom(T)$ with $(T-z)\psi=x-u$. Let $\psi_n = \hat{P}_n\psi$, then for any normalised $v_n\in\cL_n$ we have
\begin{align*}
\frak{t}(\psi_n,v_n) - z\langle \psi_n,v_n\rangle &= \frak{t}(\psi,v_n) - z\langle \psi,v_n\rangle + \frak{t}(\psi_n-\psi,v_n) - z\langle \psi_n-\psi,v_n\rangle\\
&= \langle(T-z)\psi,v_n\rangle+(\frak{t}-m)(\psi_n-\psi,v_n) + \langle \psi_n-\psi,v_n\rangle\\
&\qquad - (z-m+1)\langle \psi_n-\psi,v_n\rangle\\
&= \langle x-u,v_n\rangle+\langle \psi_n-\psi,v_n\rangle_\frak{t} - (z-m+1)\langle \psi_n-\psi,v_n\rangle\\
&= \langle x-u,v_n\rangle+\langle(\hat{P}_n-I)\psi,v_n\rangle_\frak{t} - (z-m+1)\langle \psi_n-\psi,v_n\rangle\\
&= \langle x-u,v_n\rangle - (z-m+1)\langle \psi_n-\psi,v_n\rangle
\end{align*}
and
\begin{align*}
\frak{t}(\phi_n,v_n) + \langle A\phi_n,v_n\rangle - z\langle \phi_n,v_n\rangle = (\frak{s}-z)(\phi_n,v_n) \approx \langle x,v_n\rangle.
\end{align*}
Hence, we have
\begin{align*}
\frak{t}(\phi_{n_k},v_{n_k}) - z\langle \phi_{n_k},v_{n_k}\rangle \approx \langle x-u,v_{n_k}\rangle
\end{align*}
and
\begin{align*}
\frak{t}(\phi_{n_k}-\psi_{n_k},v_{n_k}) - z\langle \phi_{n_k}-\psi_{n_k},v_{n_k}\rangle \to 0.
\end{align*}
Since $\{z\}$ is a $T_n$-regular set we deduce that $\phi_{n_k}-\psi_{n_k}$ and $\frak{t}(\phi_{n_k}-\psi_{n_k},v_{n_k})$ both converge to zero. In particular, we have
$\frak{t}[\phi_{n_k}-\psi_{n_k}]\to 0$, and therefore $\Vert\phi_{n_k}-\psi_{n_k}\Vert_{\frak{t}}\to 0$, hence
\begin{align*}
\Vert\phi_{n_k}-\psi\Vert_{\frak{t}} &\le \Vert\phi_{n_k}-\psi_{n_k}\Vert_{\frak{t}}  + \Vert\psi_{n_k}-\psi\Vert_{\frak{t}}\\
&= \Vert\phi_{n_k}-\psi_{n_k}\Vert_{\frak{t}}  + \Vert(\hat{P}_{n_k}-I)\psi\Vert_{\frak{t}}\\
&\to 0.
\end{align*}
\end{proof}

\begin{theorem}\label{resolvent_bound}
Let $\Gamma\subset\mathbb{C}$ be a compact $T_n$-regular set. If $\Gamma\subseteq\rho(T+A)$ then $\Gamma$ is an $S_n$-regular set.
\end{theorem}
\begin{proof}
Suppose the assertion is false. Then there exists a subsequence $n_k$ and a sequence $(z_{n_k})\subset\Gamma$, such that $\sigma_{n_k}(z_{n_k})\to 0$ as $k\to\infty$. We assume without loss
of generality that $\sigma_{n}(z_{n})\to 0$ for some sequence $(z_{n})\subset\Gamma$. Since $\Gamma$ is a compact set it follows that $(z_{n})$ has a
convergent subsequence, and without loss of generality we assume that $z_n\to z\in\Gamma$. Therefore, for a sequence of normalised vectors $\phi_n\in\cL_n$ we have
\[\max_{\efrac{v\in\cL_n}{\Vert v\Vert=1}}\vert(\frak{s}-z)(\phi_n,v)\vert\to 0\quad\textrm{as}\quad n\to\infty.\]
By Lemma \ref{Alemma} there exists a $u\in\mathcal{H}$ and subsequence $n_k$ with
\begin{displaymath}
\Vert A\phi_{n_k}-u\Vert\to0\quad\textrm{and}\quad\Vert\phi_{n_k}+(T-z)^{-1}u\Vert_{\frak{t}}\to0.
\end{displaymath}
Without loss of generality we assume that
\begin{displaymath}
\Vert A\phi_n-u\Vert\to0\quad\textrm{and}\quad\Vert\phi_n+(T-z)^{-1}u\Vert_{\frak{t}}\to0.
\end{displaymath}
We note that $u\ne0$. We have $z\in\rho(T+A)$ and therefore $\overline{z}\in\rho(T+A^*)$. Let $\psi\in\Dom(T+A^*)$ be such that $(T+A^*-\overline{z})\psi = -(T-z)^{-1}u$. Set $\psi_n=\hat{P}_n\psi$, then
$\Vert\psi_n-\psi\Vert_{\frak{t}}\to0$, hence $\frak{t}(\phi_n,\psi_n)\to-\frak{t}((T-z)^{-1}u,\psi)$ (see \cite[Theorem VI.1.12]{katopert}).
We obtain
\begin{align*}
0&\leftarrow(\frak{s}-z)(\phi_n,\psi_n)\\
&=\frak{t}(\phi_n,\psi_n) + \langle A\phi_n,\psi_n\rangle - z\langle \phi_n,\psi_n\rangle\\
&\to
-\frak{t}((T-z)^{-1}u,\psi) - \langle A(T-z)^{-1}u,\psi\rangle + z\langle(T-z)^{-1}u,\psi\rangle\\
&=-\langle(T-z)^{-1}u,T\psi\rangle - \langle(T-z)^{-1}u,A^*\psi\rangle + z\langle(T-z)^{-1}u,\psi\rangle\\
&=-\langle(T-z)^{-1}u,(T+A^*-\overline{z})\psi\rangle\\
&=\Vert(T-z)^{-1}u\Vert^2,
\end{align*}
however, the right hand side is non-zero. From the contradiction we deduce that $\Gamma$ is an $S_n$-regular set.
\end{proof}

\subsection{Uniform Sets}We say that an open set $U\subseteq\mathbb{C}$ is a $T_n$-uniform set
if:
\newcounter{counter_assump}
\begin{list}{{\rm\textbf{(\arabic{counter_assump})}}}%
{\usecounter{counter_assump}
\setlength{\itemsep}{0.5ex}\setlength{\topsep}{1.1ex}
\setlength{\leftmargin}{7ex}\setlength{\labelwidth}{7ex}}
\item any compact subset of $U\cap\rho(T)$ is $T_n$-regular
\item $U\cap\sigma(T,\cL_\infty)=U\cap\sigma(T)\subset\sigma_{\dis}(T)$
\item if $\lambda\in\sigma(T)\cap U$ and $\Gamma\subset U$ is a circle with center $\lambda$
and which neither intersects nor encloses any other element from $\sigma(T)$, then for all sufficiently large $n$ the total multiplicity
of those eigenvalues of $T_n$ enclosed by $\Gamma$ equals the multiplicity of the eigenvalue $\lambda$.\\
\end{list}
$U$ is assumed to be a $T_n$-uniform set for the remainder of this section. For a $\lambda\in\sigma(T)\cap U$ we denote the corresponding spectral subspace by $\cL(\{\lambda\})$. Let $\Gamma\subset U$ be a circle with center $\lambda$
and which neither intersects nor encloses any other element from $\sigma(T)$. We denote by $\cL_n(\Gamma)$
the spectral subspace associated to those elements from $\sigma(T,\cL_n)$ which are enclosed by $\Gamma$.

We use the following notions of the gap between subspaces $\cL$ and $\mathcal{M}$:
\[\delta(\cL,\mathcal{M}) = \sup_{\efrac{x\in\cL}{\Vert x\Vert=1}}\dist[x,\mathcal{M}]\quad\textrm{and}\quad\hat{\delta}(\cL,\mathcal{M})=\max\{\delta(\cL,\mathcal{M}),\delta(\mathcal{M},\cL)\},\]
see \cite[Section IV.2]{katopert} for further details. We shall write $\hat{\delta}_{\frak{t}}$ and $\delta_{\frak{t}}$ when the norm employed is $\Vert\cdot\Vert_{\frak{t}}$.

\begin{lemma}\label{subspacegap}
$\hat{\delta}_{\frak{t}}(\cL(\lambda),\cL_n(\Gamma))
=\mathcal{O}(\delta_{\frak{t}}(\cL(\lambda),\cL_n))$.
\end{lemma}
\begin{proof}
Set $\delta_{\frak{t}}(\cL(\lambda),\cL_n)=\varepsilon_n$, $\tilde{m}=m-1$ and $\tilde{\cL}_n:=(T-\tilde{m})^{\frac{1}{2}}\cL_n$.
Then for any $\phi\in\cL({\lambda})$ with $\Vert\phi\Vert_{\frak{t}}=1$ there exists a
$\psi\in\cL_n$ such that $\Vert\phi-\psi\Vert_{\frak{t}}\le\varepsilon_n$. Noting that $\Vert\phi\Vert^2=1/(\lambda-\tilde{m})$, we obtain
\begin{align*}
\varepsilon_n^2&\ge\Vert\phi - \psi\Vert_{\frak{t}}^2\\
&=(\frak{t}-m)[\phi - \psi] + \Vert\phi - \psi\Vert^2\\
&=(\frak{t}-\tilde{m})[\phi - \psi] + (1+\tilde{m}-m)\Vert\phi - \psi\Vert^2\\
&=\Vert(T-\tilde{m})^{\frac{1}{2}}(\phi - \psi)\Vert^2\\
&=\left\Vert\frac{\phi}{\Vert\phi\Vert}- (T-\tilde{m})^{\frac{1}{2}}\psi\right\Vert^2,
\end{align*}
and therefore $\delta(\cL(\lambda),\tilde{\cL}_n)\le\varepsilon_n$.

Let $\mu\in\sigma(T_n)$ with eigenvector $\psi$, then for all $v\in\cL_n$ we have
\begin{align*}
0 &= (\frak{t}-\tilde{m})(\psi,v) - (\mu-\tilde{m})\langle\psi,v\rangle\\
&=\langle(T-\tilde{m})^{\frac{1}{2}}\psi,(T-\tilde{m})^{\frac{1}{2}}v\rangle - (\mu-\tilde{m})\langle\psi,v\rangle.
\end{align*}
Setting $\tilde{\psi} =(T-\tilde{m})^{\frac{1}{2}}\psi$ and $\tilde{v} = (T-\tilde{m})^{\frac{1}{2}}v$, the above equation may be rewritten
\[
\langle(T-\tilde{m})^{-1}\tilde{\psi},\tilde{v}\rangle =
(\mu-\tilde{m})^{-1}\langle \tilde{\psi},\tilde{v}\rangle,
\]
and therefore we have the following one-to-one correspondence
between $\sigma((T-\tilde{m})^{-1},\tilde{\cL}_n)$ and $\sigma(T_n)$:
\[
\mu\in\sigma(T_n)\quad\iff\quad\frac{1}{\mu-\tilde{m}}\in\sigma((T-\tilde{m})^{-1},\tilde{\cL}_n).
\]
Let $(\tilde{P}_n)$ be the orthogonal projections from $\mathcal{H}$ onto $\tilde{\cL}_n$. Since the sequence
$(\cL_n)$ is dense in $\mathcal{H}_{\frak{t}}$ it follows that the sequence
$(\tilde{\cL}_n)$ is dense in $\mathcal{H}$, i.e. that $\tilde{P}_n\stackrel{s}{\longrightarrow} I$. Hence, $\tilde{P}_n(T-\tilde{m})^{-1}\tilde{P}_n\stackrel{s}{\longrightarrow}(T-\tilde{m})^{-1}$, $(\lambda-\tilde{m})^{-1}$ is isolated in $\sigma((T-\tilde{m})^{-1},\tilde{\cL}_\infty)$, and for all sufficiently large $n$ the total multiplicity of those elements from $\sigma((T-\tilde{m})^{-1},\tilde{\cL}_n)$ in a neighbourhood of $(\lambda-\tilde{m})^{-1}$ equals $\dim\cL(\lambda)$.
We denote by $\tilde{\cL}_n(\Gamma)$ the spectral subspace corresponding to those eigenvalues from $\sigma((T-\tilde{m})^{-1},\tilde{\cL}_n)$ which are in a neighbourhood of $(\lambda-\tilde{m})^{-1}$. Then combining the estimate from the previous paragraph with \cite[Theorem 6.6 and Lemma 6.9]{chat} we obtain
\begin{equation}\label{rescon}
\hat{\delta}(\cL(\lambda),\tilde{\cL}_n(\Gamma))=\mathcal{O}(\varepsilon_n).
\end{equation}

Let $M\ge0$ be such that $\hat{\delta}(\cL(\lambda),\tilde{\cL}_n(\Gamma))\le M\varepsilon_n$ for all $n\in\mathbb{N}$.
It follows that for some $\psi\in\cL_n(\Gamma)$ we have
\begin{displaymath}
\left\Vert\phi-\frac{(T-\tilde{m})^{\frac{1}{2}}\psi}{\sqrt{\lambda-\tilde{m}}}\right\Vert\le \frac{M\varepsilon_n}{\sqrt{\lambda-\tilde{m}}}=\Vert\phi\Vert M\varepsilon_n.
\end{displaymath}
Then
\begin{align*}
\Vert\phi-\psi\Vert^2_{\frak{t}} &= (\frak{t}-m)[\phi-\psi] + \Vert\phi-\psi\Vert^2\\
&=\Vert(T-\tilde{m})^{\frac{1}{2}}(\phi-\psi)\Vert^2 + (\tilde{m}-m+1)\Vert\phi-\psi\Vert^2\\
&=(\lambda-\tilde{m})\left\Vert\phi - \frac{(T-\tilde{m})^{\frac{1}{2}}\psi}{\sqrt{\lambda-\tilde{m}}}\right\Vert^2\\
&\le M^2\varepsilon_n^2
\end{align*}
and therefore $\delta_{\frak{t}}(\cL(\lambda),\cL_n(\Gamma))=\mathcal{O}(\varepsilon_n)$. Since $\dim\cL_n(\Gamma)=\dim\cL(\lambda)<\infty$ for all sufficiently large $n$, we have
the estimate
 \[
 \delta_{\frak{t}}(\cL_n(\Gamma),\cL(\lambda))\le\frac{\delta_{\frak{t}}(\cL(\lambda),\cL_n(\Gamma))}{1-\delta_{\frak{t}}(\cL(\lambda),\cL_n(\Gamma))};
 \]
 see \cite[Lemma 213]{kato}. We deduce that $\tilde{\delta}_{\frak{t}}(\cL(\lambda),\cL_n(\Gamma))=\mathcal{O}(\varepsilon_n)$.
\end{proof}

\begin{theorem}\label{spectralinc}
If $U$ is a $T_n$-uniform set, then $\sigma(T+A,\cL_\infty)\cap U = \sigma(T+A)\cap U$.
\end{theorem}

\begin{remark}
If the essential spectrum of $T$ is non-empty and $\lambda_{\textrm{e}}^-=\min\sigma_{\ess}(T)$ then
$\mathbb{C}\backslash[\lambda_{\textrm{e}}^-,\infty)$ is a $T_n$-uniform set (though not necessarily the largest). If the essential spectrum is
empty then $\mathbb{C}$ is a $T_n$-uniform set. If $T$ is bounded and $\lambda_{\textrm{e}}^+=\max\sigma_{\ess}(T)$ then
$\mathbb{C}\backslash(\lambda_{\textrm{e}}^-,\lambda_{\textrm{e}}^+)$ is a $T_n$-uniform set (though not necessarily the largest).
\end{remark}

\begin{proof}[Proof of Theorem \ref{spectralinc}]
First we show that $\sigma(T+A,\cL_\infty)\cap U\subset\sigma(T+A)$. Let $z\in\rho(T+A)\cap U$. First suppose that $z\in\rho(T)$. Then $\{z\}$ is a $T_n$-regular set and by Theorem \ref{resolvent_bound} also an $S_n$-regular set. We
deduce that $z\notin\sigma(T+A,\cL_\infty)$. Suppose now that $z\in\sigma(T)$ and that $z\in\sigma(T+A,\cL_\infty)$.
Therefore, for a sequence of normalised vectors $\phi_n\in\cL_n$ we have
\[\max_{\efrac{v\in\cL_n}{\Vert v\Vert=1}}\vert(\frak{s}-z)(\phi_n,v)\vert\to 0\quad\textrm{as}\quad n\to\infty.\]
By Lemma \ref{Alemma} the sequence $(A\phi_n)$ has a convergent subsequence. Without loss of generality we
assume that $A\phi_n\to u$, and hence for a normalised $v\in\cL_n$ we have
\begin{equation}\label{tcon}
\frak{t}(\phi_n,v) - z\langle\phi_n,v\rangle\approx-\langle u,v\rangle.
\end{equation}
For any $y\in\mathcal{H}$ there exists a $\psi\in\Dom(T+A^*)$ such that $(T+A^*-\overline{z})\psi=y$. Let $\psi_n=\hat{P}_n\psi$, then
\begin{align*}
0&\leftarrow(\frak{s}-z)(\phi_n,\psi_n)\\
&=\frak{t}(\phi_n,\psi_n) + \langle A\phi_n,\psi_n\rangle - z\langle \phi_n,\psi_n\rangle\\
&=\frak{t}(\phi_n,\psi_n-\psi) + \langle A\phi_n,\psi_n-\psi\rangle - z\langle \phi_n,\psi_n-\psi\rangle + \langle\phi_n,y\rangle\\
&= \langle\phi_n,\psi_n-\psi\rangle_{\frak{t}} + \langle A\phi_n,\psi_n-\psi\rangle - (z-m+1)\langle \phi_n,\psi_n-\psi\rangle + \langle\phi_n,y\rangle\\
&= \langle\phi_n,(\hat{P}_n-I)\psi\rangle_{\frak{t}} + \langle A\phi_n,\psi_n-\psi\rangle - (z-m+1)\langle \phi_n,\psi_n-\psi\rangle + \langle\phi_n,y\rangle\\
&= \langle A\phi_n,\psi_n-\psi\rangle - (z-m+1)\langle \phi_n,\psi_n-\psi\rangle + \langle\phi_n,y\rangle\\
&\approx \langle\phi_n,y\rangle,
\end{align*}
and therefore $\phi_n\stackrel{w}{\longrightarrow}0$. For an arbitrary $\psi\in\Dom(T)$ let $\psi_n=\hat{P}_n\psi$, then
\begin{align*}
&=\frak{t}(\phi_n,\psi) - z\langle \phi_n,\psi\rangle - \frak{t}(\phi_n,\psi_n) + z\langle \phi_n,\psi_n\rangle\\
&=(m-1-z)\langle \phi_n,\psi-\psi_n\rangle\\
&\to 0.
\end{align*}
Then using \eqref{tcon} it follows that
$-\langle u,\psi\rangle \leftarrow \frak{t}(\phi_n,\psi) - z\langle \phi_n,\psi\rangle = \langle\phi_n,(T-z)\psi\rangle\to 0$,
which implies that $u=0$. Therefore, we have
\[\max_{\efrac{v\in\cL_n}{\Vert v\Vert=1}}\vert(\frak{t}-z)(\phi_n,v)\vert\to 0\quad\textrm{as}\quad n\to\infty,\]
which together with Lemma \ref{subspacegap} implies that $\dist(\phi_n,\ker(T-z))\to 0$
which is a contradiction since $\ker(T-z)$ is finite dimensional and $\phi_n$ converges weakly to zero. We deduce that
$\sigma(T+A,\cL_\infty)\cap U\subset\sigma(T+A)$.

It remains to show that $\sigma(T+A)\cap U\subset\sigma(T+A,\cL_\infty)$.
Let $z\in\sigma(T+A)\cap U$. Then $z\in\sigma_{\dis}(T+A)$ and we choose a circle $\Gamma$ contained
in $(\rho(T+A)\cup\rho(T))\cap U$ with center $z$ and which encloses no other element from $\sigma(T+A)$. By Theorem \ref{resolvent_bound}, $\Gamma$ is an $S_n$-regular set.
Let $(T+A-z)\phi=0$ and $\phi_n=\hat{P}_n\phi$, then $\Vert \phi-\phi_n\Vert_{\frak{t}}\to0$ and $\frak{t}[\phi_n]\to\frak{t}[\phi]$. It follows that $(\vert T\vert^{\frac{1}{2}}\phi_n)$ is a bounded sequence and therefore that $A\phi_{n_j}\to A\phi$ for some subsequence $n_j$. We show that in fact $A\phi_n\to A\phi$.
Suppose the contrary, then there exists a subsequence $n_k$ and a $\delta>0$ such that $\Vert A\phi_{n_k}-A\phi\Vert\ge\delta$ for all $k\in\mathbb{N}$. However, $(\vert T\vert^{\frac{1}{2}}\phi_{n_k})$ is a bounded sequence and therefore $A\phi_{n_k}$ has a convergent subsequence which must converge to $A\phi$. From the
contradiction we deduce that $A\phi_n\to A\phi$. We have
\begin{align*}
\sigma_n(z)&\le\max_{\efrac{v\in\cL_n}{\Vert v\Vert=1}}
\vert(\frak{s}-z)(\phi_n,v)\vert\\
&=\max\vert\frak{t}(\phi_n,v) + \langle A\phi_n,v\rangle  - z\langle\phi_n,v\rangle\vert\\
&=\max\vert\frak{t}(\phi_n-\phi,v) + \langle A(\phi_n-\phi),v\rangle  - z\langle\phi_n-\phi,v\rangle\vert\\
&=\max\vert\langle(\hat{P}_n-I)\phi,v\rangle_{\frak{t}} + \langle A(\phi_n-\phi),v\rangle- (z-m+1)\langle\phi_n-\phi,v\rangle\vert\\
&=\max\vert\langle A(\phi_n-\phi),v\rangle  - (z+m-1)\langle\phi_n-\phi,v\rangle\vert\\
&\to 0.
\end{align*}
Hence, for all sufficiently large $n$, the function $\sigma_n(\cdot)$ has a local minimum inside the circle $\Gamma$ and therefore $\sigma(T+A,\cL_n)$
intersects the interior of $\Gamma$; see \cite[Theorem 9.2.8]{EBD}. The radius of $\Gamma$ may
be chosen arbitrarily small from which we deduce that $z\in\sigma(T+A,\cL_\infty)$, as required.
\end{proof}

\void{\begin{lemma}\label{kerdist}
Let $z\in\sigma(T+iQ)\backslash\mathbb{R}$ and $(S_n - z_n)\phi_n=0$ where $\Vert\phi_n\Vert=1$ and $z_n\to z$, then
\begin{displaymath}
\dist_{\frak{t}}\big(\phi_n,\Ker(T+A-z)\big)\to 0.
\end{displaymath}
\end{lemma}
\begin{proof}
Clearly we have $(S_n - z)\phi_n\to0$, then by Lemma \ref{Alemma} there exists a subsequence $n_k$ and $u\in\mathcal{H}$, such that
\begin{displaymath}
\Vert A\phi_{n_k} - u\Vert\to 0\quad\textrm{and}\quad\Vert\phi_{n_k}+(T-z)^{-1}u\Vert_{\frak{t}}\to 0.
\end{displaymath}
We show that $(T+A-z)(T-z)^{-1}u=0$. Let $v\in\Dom(\frak{t})$, then
\begin{align*}
0&\leftarrow\langle(S_{n_k}-z)\phi_{n_k},v\rangle\\
&=\frak{t}(\phi_{n_k},v) + \langle A\phi_{n_k},v\rangle - z\langle\phi_{n_k},v\rangle\\
&\to-\frak{t}((T-z)^{-1}u,v) - \langle A(T-z)^{-1}u,v\rangle + z\langle(T-z)^{-1}u,v\rangle\\
&=-\langle(T+A - z)(T-z)^{-1}u,v\rangle.
\end{align*}
Suppose there exists a subsequence $n_j$ such that
\begin{displaymath}
\max\big\{\Vert\phi_{n_j} - y\Vert_{\frak{t}}:y\in\Ker(T+A-z)\big\}>\delta>0\quad\textrm{for every}\quad j\in\mathbb{N}.
\end{displaymath}
Arguing as above, there exists a subsequence $n_{j_k}$ and a $y\in\Ker(T+A-z)$ such that $\Vert\phi_{n_{j_k}} - y\Vert_{\frak{t}}\to 0$.
The result follows from the contradiction
\end{proof}}

\subsection{Multiplicity}\label{secmult}We consider an eigenvalue $z\in\sigma(T+A)\cap U$
where $U$ is a $T_n$-uniform set. We denote by $\Gamma$ a circle contained in $U$ with
center $z$ and which neither intersects nor encloses any additional element from $\sigma(T+A)\cup\sigma(T)$.
The spectral projections associated to the part of $\sigma(T+A)$ and $\sigma(T+A,\cL_n)$
enclosed by the circle $\Gamma$ are defined by
\begin{align*}
P(z)&:=-\frac{1}{2i\pi}\int_{\Gamma}(T+A-\zeta)^{-1}~d\zeta\quad\textrm{and}\\
P_n(\Gamma)&:=-\frac{1}{2i\pi}\int_{\Gamma}(S_n-\zeta)^{-1}~d\zeta,
\end{align*}
respectively. We denote by $\mathcal{M}(z)$, $\mathcal{K}(z)$, and $\mathcal{M}_n(\Gamma)$ the range of $P(z)$, $I-P(z)$, and $P_n(\Gamma)$, respectively. We denote by $\cL_n(\Gamma)$ the spectral subspace associated to those elements from $\sigma(T,\cL_n)$ which are enclosed by $\Gamma$.

We introduce the operator $\mathcal{T}$ with domain $\Dom(\mathcal{T})=\Dom(T)\ominus\ker(T-z)$ and action $\mathcal{T}\phi = T\phi$. Evidently, $\mathcal{T}$ is
a self-adjoint operator on the Hilbert space $\mathcal{H}\ominus\ker(T-z)$ and we have $z\in\rho(\mathcal{T})$. We do not assume that $z\in\sigma(T)$, so that in the case
where $z\in\rho(T)$ we have $\mathcal{T}=T$ and $\mathcal{H}\ominus\ker(T-z)=\mathcal{H}$.

\begin{lemma}\label{Blemma}
Let $(\phi_n)$ be a bounded sequence of vectors with $\phi_n\in\cL_n$ and
\[\max_{\efrac{v\in\cL_n}{\Vert v\Vert=1}}\vert(\frak{s}-z)(\phi_n,v) - \langle x,v\rangle\vert\to 0\quad\textrm{as}\quad n\to\infty.\]
There exists a $u\in\mathcal{H}$ and subsequence $n_k$, such that $\Vert A\phi_{n_k} - u\Vert\to 0$. Moreover, $x-u\perp\ker(T-z)$ and $\dist_{\frak{t}}[\phi_{n_k}-(\mathcal{T}-z)^{-1}(x-u),\ker(T-z)]\to 0$.
\end{lemma}
\begin{proof}
For the first statement see Lemma \ref{Alemma}. If $z\in\rho(T)$ then for the second statement see Lemma \ref{Alemma}.
It remains to consider the case where $z\in\sigma(T)$. Let $\phi\in\ker(T-z)$, then
\begin{align*}
\langle x-u,\phi\rangle \approx  \frak{t}(\phi_{n_k},\phi) - z\langle\phi_{n_k},\phi\rangle = 0,
\end{align*}
and therefore $x-u\perp\ker(T-z)$.
There exists a vector $\psi\in\Dom(T)\ominus\ker(T-z)$ with $(\mathcal{T}-z)\psi=(T-z)\psi=x-u$. Let $\psi_n = \hat{P}_n\psi$ and $v_n\in\cL_n$ with $\Vert v_n\Vert=1$, then arguing precisely as in Lemma \ref{Alemma} we have
\begin{align*}
&\frak{t}(\psi_n,v_n) - z\langle \psi_n,v_n\rangle=\langle x-u,v_n\rangle - (z-m+1)\langle \psi_n-\psi,v_n\rangle\quad\textrm{and}\\
&\frak{t}(\phi_n,v_n) - z\langle \phi_n,v_n\rangle= \langle(S_n-z)\phi_n,v_n\rangle - \langle A\phi_n,v_n\rangle \approx \langle x,v_n\rangle - \langle A\phi_n,v_n\rangle.
\end{align*}
Hence, we have $\frak{t}(\phi_{n_k}-\psi_{n_k},v_{n_k}) - z\langle \phi_{n_k}-\psi_{n_k},v_{n_k}\rangle \to 0$,
from which we deduce that $\dist[\phi_{n_k}-\psi_{n_k},\cL_{n_k}(\Gamma)]\to 0$. Let $x_k\in\cL_{n_k}(\Gamma)$ be such that
$\phi_{n_k}-\psi_{n_k}-x_k\to 0$, and note that $(\phi_{n_k}-\psi_{n_k})$ and therefore also $(x_k)$ is a bounded sequence in $\mathcal{H}$. Furthermore, if
$\phi_{n_k}-\psi_{n_k}\to 0$ then we may choose $x_k=0$ for every $k$.
If $d=\dim\ker(T-z)$, then $\dim\cL_{n_k}(\Gamma)=d$ for all sufficiently large $k$. Hence $\cL_{n_k}(\Gamma)=\spn\{y_{k,1},\dots,y_{k,d}\}$ where the
$y_{k,j}$ are orthonormal and
\[\frak{t}(y_{k,j},v)=z_{k,j}\langle y_{k,j},v\rangle\quad\forall v\in\cL_{n_k}\quad\textrm{where}\quad z_{k,j}\approx z.\]
Hence $x_k = \sum\alpha_{k,j}y_{k,j}$ and
\begin{align*}
\Vert\phi_{n_k}-\psi_{n_k}-x_k\Vert^2_{\frak{t}} &=
(\frak{t}-z)[\phi_{n_k}-\psi_{n_k}-x_k] + (z-m+1)\Vert\phi_{n_k}-\psi_{n_k}-x_k\Vert^2\\
&=(\frak{t}-z)[\phi_{n_k}-\psi_{n_k}] + \sum\overline{\alpha}_{k,j}(\frak{t}-z)(\phi_{n_k}-\psi_{n_k},y_{k,j})\\
&\quad + \sum\alpha_{k,j}(\frak{t}-z)(y_{k,j},\phi_{n_k}-\psi_{n_k})\\
&\quad + (\frak{t}-z)\sum\overline{\alpha}_{k,i}\alpha_{k,j}(y_{k,j},y_{k,i})\\
&\quad + (z-m+1)\Vert\phi_{n_k}-\psi_{n_k}-x_k\Vert^2\\
&=(\frak{t}-z)[\phi_{n_k}-\psi_{n_k}] + \sum\overline{\alpha}_{k,j}(z_{k,j}-z)(\phi_{n_k}-\psi_{n_k},y_{k,j})\\
&\quad + \sum\alpha_{k,j}(z_{k,j}-z)(y_{k,j},\phi_{n_k}-\psi_{n_k})\\
&\quad + \sum\vert\alpha_{k,j}\vert^2(z_{k,j}-z)\\
&\quad + (z-m+1)\Vert\phi_{n_k}-\psi_{n_k}-x_k\Vert^2\\
&\to 0.
\end{align*}
Using Lemma \ref{subspacegap}
it follows that there exists a sequence $y_k\in\ker(T-z)$ such that $\Vert x_k-y_k\Vert_{\frak{t}}\to 0$, and therefore
\begin{align*}
\Vert\phi_{n_k}-(\mathcal{T}-z)^{-1}(x-u)-y_k\Vert_{\frak{t}}&=\Vert\phi_{n_k}-\psi-y_k\Vert_{\frak{t}}\\
&\le \Vert\phi_{n_k}-\psi_{n_k}-x_k\Vert_{\frak{t}} + \Vert \psi_{n_k}-\psi\Vert_{\frak{t}} + \Vert x_k-y_k\Vert_{\frak{t}}\\
&\to 0.
\end{align*}
\end{proof}

\begin{lemma}\label{mult}
If $\mathcal{M}(z)\subset\cL_n$ for all $n\in\mathbb{N}$, then $\mathcal{M}_n(\Gamma)=\mathcal{M}(z)$ for all sufficiently large $n$.
\end{lemma}
\begin{proof}
Evidently, we have $z\in\sigma(T+A,\cL_n)$ for every $n\in\mathbb{N}$. We denote by $\mathcal{M}_n(z)$ the spectral subspace associated to $S_n$ and the eigenvalue $z\in\sigma(T+A,\cL_n)$. We note that $\mathcal{M}(z)\subseteq\mathcal{M}_n(z)\subseteq\mathcal{M}_n(\Gamma)$ for all $n$.

Suppose that $\mathcal{M}(z)\subsetneqq\mathcal{M}_{n_k}(z)$ for some subsequence $n_k$, and without loss of
generality we assume that $\mathcal{M}(z)\subsetneqq\mathcal{M}_{n}(z)$ for every $n\in\mathbb{N}$.
Then we may choose a normalised
sequence $(\phi_n)$ with $\phi_n\in\mathcal{M}_n(z)$ and
\[\phi_n\in\mathcal{K}(z)\quad\textrm{and}\quad x_n:=(S_n-z)\phi_n\in\mathcal{M}(z).
\]
To see this, we note that there is at least one vector
$y\in\mathcal{K}(z)\cap\mathcal{M}_n(z)$, therefore
\begin{align*}
P(z)(S_n-z)&y  + (I-P(z))(S_n-z)y = (S_n-z)y  \in\mathcal{M}_n(z)\\
&\Rightarrow(I-P(z))(S_n-z)y \in\mathcal{M}_n(z).
\end{align*}

First consider the case where $x_{n_k}\to 0$ for some subsequence $n_k$. We assume without loss
of generality that $x_n\to 0$. Using Lemma \ref{Blemma} and the fact that $0\le\dim\ker(T-z)<\infty$, we have a $u\in\mathcal{H}$, a subsequence $n_k$, and $\phi\in\ker(T-z)$, such that
\begin{displaymath}
\Vert A\phi_{n_k} - u\Vert\to 0\quad\textrm{and}\quad\Vert\phi_{n_k}+(\mathcal{T}-z)^{-1}u - \phi\Vert_{\frak{t}}\to 0.
\end{displaymath}
We assume without loss of generality that
\begin{displaymath}
\Vert A\phi_{n} - u\Vert\to 0\quad\textrm{and}\quad\Vert\phi_{n}+(\mathcal{T}-z)^{-1}u - \phi\Vert_{\frak{t}}\to 0.
\end{displaymath}
We note that $\phi_{n}\to-(\mathcal{T}-z)^{-1}u + \phi$ implies that
\begin{displaymath}
(\mathcal{T}-z)^{-1}u - \phi\in\mathcal{K}(z)\quad\textrm{and}\quad A\phi_n\to -A(\mathcal{T}-z)^{-1}u + A\phi.
\end{displaymath}
Let $v\in\Dom(T+A^*)$ and $v_n=\hat{P}_nv$, then
\begin{align*}
0&\leftarrow\langle(S_n-z)\phi_n,v_n\rangle\\
&=\frak{t}(\phi_n,v_n) + \langle A\phi_n,v_n\rangle - z\langle \phi_n,v_n\rangle\\
&\to-\frak{t}((\mathcal{T}-z)^{-1}u - \phi,v) - \langle A(\mathcal{T}-z)^{-1}u - A\phi,v\rangle + z\langle (\mathcal{T}-z)^{-1}u - \phi,v\rangle\\
&=-\langle(\mathcal{T}-z)^{-1}u-\phi,Tv\rangle - \langle(\mathcal{T}-z)^{-1}u-\phi,A^*v\rangle + z\langle (\mathcal{T}-z)^{-1}u-\phi,v\rangle\\
&=-\langle(\mathcal{T}-z)^{-1}u-\phi,(T+A^*-\overline{z})v\rangle.
\end{align*}
It follows that $(\mathcal{T}-z)^{-1}u-\phi\in\Ker(T+A-z)\subseteq\mathcal{M}(z)$, and we obtain a contradiction since $(\mathcal{T}-z)^{-1}u-\phi\in\mathcal{K}(z)$.

We suppose now that $\Vert x_n\Vert\ge c> 0$ for all sufficiently large $n\in\mathbb{N}$. Let $\hat{\phi}_n = \phi_n/\Vert x_n\Vert$, then
since $\mathcal{M}(z)$ is finite dimensional we have for some subsequence $n_k$
\begin{displaymath}
(S_{n_k} - z)\hat{\phi}_{n_k}\to x\quad\textrm{for some normalised}\quad x\in\mathcal{M}(z).
\end{displaymath}
We assume without loss of generality that $(S_{n} - z)\hat{\phi}_{n}\to x$.
Using Lemma \ref{Blemma} as above, we may assume that
\begin{displaymath}
\Vert A\hat{\phi}_{n} - u\Vert\to 0\quad\textrm{and}\quad\Vert\hat{\phi}_{n}-(\mathcal{T}-z)^{-1}(x-u)-\phi\Vert_{\frak{t}}\to 0\quad\textrm{where}\quad\phi\in\ker(T-z).
\end{displaymath}
We note that $\hat{\phi}_{n}\to(\mathcal{T}-z)^{-1}(x-u)+\phi$ implies that
\begin{displaymath}
(\mathcal{T}-z)^{-1}(x-u)+\phi\in\mathcal{K}(z)\quad\textrm{and}\quad A\hat{\phi}_{n} \to A(\mathcal{T}-z)^{-1}(x-u) + A\phi.
\end{displaymath}
We have
\begin{displaymath}
(T+A-z)\big((\mathcal{T}-z)^{-1}(x-u)+\phi\big) = x - u + A(\mathcal{T}-z)^{-1}(x-u) + A\phi = x\in\mathcal{M}(z)
\end{displaymath}
which is a contradiction since $(\mathcal{T}-z)^{-1}(x-u)+\phi\in\mathcal{K}(z)$.

We have shown that $\mathcal{M}(z)=\mathcal{M}_{n}(z)$ for all sufficiently large $n$. It remains to show that $z$ is not the limit point
of a sequence $z_{n_k}\in\sigma(T+A,\cL_{n_k})$ where $z_{n_k}\ne z$ for each $k\in\mathbb{N}$.
Suppose the contrary and without loss of generality that $(S_n -z_n)\phi_n = 0$ for some normalised vectors $\phi_n\in\cL_n$ where $z_n\to z$, and $z_n\ne z$ for each $n\in\mathbb{N}$. Therefore $(S_n -z)\phi_n\to0$, and using Lemma \ref{Blemma} as above, we may assume that
\begin{displaymath}
\Vert A\phi_n - u\Vert\to 0\quad\textrm{and}\quad\Vert\phi_n+(\mathcal{T}-z)^{-1}u - \phi\Vert_{\frak{t}}\to 0\quad\textrm{where}\quad\phi\in\ker(T-z).
\end{displaymath}
Arguing as above, we let $v\in\Dom(T+A^*)$ and $v_n=\hat{P}_nv$, then
\begin{align*}
0\leftarrow\langle(S_n-z)\phi_n,v_n\rangle
\to-\langle(\mathcal{T}-z)^{-1}u-\phi,(T+A^*-\overline{z})v\rangle.
\end{align*}
It follows that $(\mathcal{T}-z)^{-1}u-\phi\in\Ker(T+A-z)\subseteq\mathcal{M}(z)$, and therefore that $\phi_n \approx \psi_n\in\Ker(T+A-z)$ where $\Vert\psi_n\Vert=1$. We have
\begin{equation}\label{decom}
0 = (S_n-z_n)\phi_n = (S_n-z_n)(I-P(z))\phi_n + (S_n-z_n)P(z)\phi_n,
\end{equation}
with $0\ne(I-P(z))\phi_n\to 0$ and $(S_n-z_n)P(z)\phi_n\to0$. To see the latter let $v_n\in\cL_n$ and write $P(z)\phi_n = \psi_n + \varepsilon_ny_n$
where $\Vert y_n\Vert=1$ and $\varepsilon_n\to 0$, then
\begin{align*}
\langle(S_n - z_n)P(z)\phi_n,v_n\rangle &= \frak{t}(P(z)\phi_n,v_n) + \langle AP(z)\phi_n,v_n\rangle - z_n\langle P(z)\phi_n,v_n\rangle\\
&=\langle(T + A - z_n)P(z)\phi_n,v_n\rangle\\
&=(z-z_n)\langle\psi_n,v_n\rangle + \varepsilon_n\langle(T + A - z_n)P(z)y_n,v_n\rangle
\end{align*}
that the first term on the right hand side converges to zero is clear, for the second term we note that $(T + A - z_n)P(z)$ is a bounded operator.
We denote
\begin{displaymath}
\hat{\phi}_n = \frac{(I-P(z))\phi_n}{\Vert(S_n-z_n)P(z)\phi_n\Vert}\in\mathcal{K}(z),
\end{displaymath}
then using \eqref{decom} we have for some subsequence $n_k$
\begin{displaymath}
(S_{n_k}-z_{n_k})\hat{\phi}_{n_k} = -\frac{(S_{n_k}-z_{n_k})P(z)\phi_{n_k}}{\Vert(S_{n_k}-z_{n_k})P(z)\phi_{n_k}\Vert}\to x\quad\textrm{for some normalised}\quad x\in\mathcal{M}(z).
\end{displaymath}
We assume without loss of generality that $(S_n-z_n)\hat{\phi}_n \to x\in\mathcal{M}(z)$.

First consider the case where the sequence $\Vert\hat{\phi}_n\Vert$ is bounded. Using Lemma \ref{Blemma} as above, we may assume that
\begin{displaymath}
\Vert A\hat{\phi}_n - u\Vert\to 0\quad\textrm{and}\quad\Vert\hat{\phi}_n-(\mathcal{T}-z)^{-1}(x-u)-\phi\Vert_{\frak{t}}\to 0\quad\textrm{where}\quad\phi\in\ker(T-z).
\end{displaymath}
We note that $\hat{\phi}_{n}\to(\mathcal{T}-z)^{-1}(x-u)+\phi$ implies that
\begin{displaymath}
(\mathcal{T}-z)^{-1}(x-u)+\phi\in\mathcal{K}(z)\quad\textrm{and}\quad A\hat{\phi}_n \to A(\mathcal{T}-z)^{-1}(x-u) + A\phi.
\end{displaymath}
We have
\begin{displaymath}
(T+A-z)\big((\mathcal{T}-z)^{-1}(x-u)+\phi\big) = x - u + A(\mathcal{T}-z)^{-1}(x-u) +A\phi= x\in\mathcal{M}(z)
\end{displaymath}
which is a contradiction since $(\mathcal{T}-z)^{-1}(x-u)+\phi\in\mathcal{K}(z)$.

Suppose now that the sequence $\Vert\hat{\phi}_n\Vert$ is not bounded. In view of the previous paragraph we assume that $\Vert\hat{\phi}_n\Vert\to\infty$. We set
$\tilde{\phi}_n = \hat{\phi}_n/\Vert\hat{\phi}_n\Vert$ and obtain
\begin{displaymath}
(S_n-z_n)\tilde{\phi}_n  \to0\quad\Rightarrow\quad (S_n-z)\tilde{\phi}_n\to 0.
\end{displaymath}
Then using Lemma \ref{Blemma} as above, we may assume that
\begin{displaymath}
\Vert A\tilde{\phi}_n - u\Vert\to 0\quad\textrm{and}\quad\Vert\tilde{\phi}_n+(\mathcal{T}-z)^{-1}u-\phi\Vert_{\frak{t}}\to 0\quad\textrm{where}\quad\phi\in\ker(T-z).
\end{displaymath}
We note that $\tilde{\phi}_{n}\to-(\mathcal{T}-z)^{-1}u+\phi$ implies that
\begin{displaymath}
(\mathcal{T}-z)^{-1}u-\phi\in\mathcal{K}(z)\quad\textrm{and}\quad A\tilde{\phi}_n \to -A(\mathcal{T}-z)^{-1}u+A\phi.
\end{displaymath}
Let $v\in\Dom(T+A^*)$ and $v_n=\hat{P}_nv$, then
\begin{align*}
0&\leftarrow\langle(S_n-z)\tilde{\phi}_n,v_n\rangle\\
&=\frak{t}(\tilde{\phi}_n,v_n) + \langle A\tilde{\phi}_n,v_n\rangle - z\langle \tilde{\phi}_n,v_n\rangle\\
&\to-\frak{t}((\mathcal{T}-z)^{-1}u-\phi,v) - \langle A(\mathcal{T}-z)^{-1}u-A\phi,v\rangle + z\langle (\mathcal{T}-z)^{-1}u-\phi,v\rangle\\
&=-\langle(\mathcal{T}-z)^{-1}u-\phi,Tv\rangle - \langle(\mathcal{T}-z)^{-1}u-\phi,A^*v\rangle + z\langle (\mathcal{T}-z)^{-1}u-\phi,v\rangle\\
&=-\langle(\mathcal{T}-z)^{-1}u-\phi,(T+A^*-\overline{z})v\rangle.
\end{align*}
It follows that $(\mathcal{T}-z)^{-1}u-\phi\in\Ker(T+A-z)\subseteq\mathcal{M}(z)$, and we obtain a contradiction since $(\mathcal{T}-z)^{-1}u-\phi\in\mathcal{K}(z)$.
\end{proof}

\begin{theorem}\label{multthm}
$\rank~P_n(\Gamma) = \rank~P(z)$ for all sufficiently large $n\in\mathbb{N}$.
\end{theorem}
\begin{proof}
Let $\Span\{\phi_1,\dots,\phi_d\}=\mathcal{M}(z)$ where the $\phi_j$ are orthonormal in $\mathcal{H}$, and set
$\phi_{n,j} = \hat{P}_n\phi_j$. Note that there exists a sequence $\varepsilon_n\to0$ such that
\begin{equation}\label{ep}
\Vert(I-\hat{P}_n)\phi_j\Vert\le\varepsilon_n,
\quad\textrm{and}\quad\vert\langle \phi_{n,j},\phi_{n,i}\rangle\vert\le\varepsilon_n\quad\textrm{for}\quad i\ne j.
\end{equation}
Evidently, the vectors $\{\phi_{n,1},\dots,\phi_{n,d}\}$ form a linearly independent set for all sufficiently large $n\in\mathbb{N}$.
There exist orthonormal vectors $\phi_{n,d+1},\dots,\phi_{n,n}$ (we assume without loss of generality that $\dim \cL_n = n$) such that
\begin{equation}\label{perp1}
\cL_n=\Span\{\phi_{n,1},\dots,\phi_{n,d},\phi_{n,d+1},\dots,\phi_{n,n}\}
\end{equation}
and
\begin{equation}\label{perp2}
\Span\{\phi_{n,1},\dots,\phi_{n,d}\}\perp\Span\{\phi_{n,d+1},\dots,\phi_{n,n}\},
\end{equation}
where on both occasions the orthogonality is with respect to $\mathcal{H}$.
For $1\le j\le d$ we set $\phi_{n,j}(t) = t\phi_j + (1-t)\phi_{n,j}$ where $t\in\mathbb{R}$. Let $\underline{\alpha}\in\mathbb{C}^n$, then
\begin{align*}
\sum_{j=1}^d\alpha_j\phi_{n,j}(t) + \sum_{j=d+1}^n\alpha_j\phi_{n,j} &=
\sum_{j=1}^d\alpha_j(t\phi_j + (1-t)\phi_{n,j}) + \sum_{j=d+1}^n\alpha_j\phi_{n,j}\\
&=\sum_{j=1}^n\alpha_j\phi_{n,j}+\sum_{j=1}^d\alpha_j t(\phi_j -\phi_{n,j})\\
&=\sum_{j=1}^n\alpha_j\phi_{n,j} + (I-\hat{P}_n)\sum_{j=1}^d\alpha_j t\phi_j.
\end{align*}
The two summations on the right hand side are orthogonal in $\mathcal{H}_{\frak{t}}$. Hence the left hand side can only vanish if both terms on the right hand side vanish, that is, if $\underline{\alpha}=\underline{0}$. We deduce that the vectors
$\{\phi_{n,1}(t),\dots,\phi_{n,d}(t),\phi_{n,d+1},\dots,\phi_{n,m}\}$ form a linearly independent set for every $t\in\mathbb{R}$.
We define the family of $n$-dimensional subspaces
\begin{displaymath}
\cL_n(t):=\Span\{\phi_{n,1}(t),\dots,\phi_{n,d}(t),\phi_{n,d+1},\dots,\phi_{n,n}\}.
\end{displaymath}
For any $\psi\in\Dom(\frak{t})$, there exist vectors $\psi_n\in\cL_n$ such that $\Vert\psi-\psi_n\Vert_{\frak{t}}\to 0$.
Using \eqref{perp1} we have
\begin{align*}
\psi_n &= \sum_{j=1}^n\alpha_{n,j}\phi_{n,j}\quad\textrm{for some}\quad\underline{\alpha}\in\mathbb{C}^n.
\end{align*}
Note that for some $M\in\mathbb{R}$ we have $\Vert\psi_n\Vert^2\le M$ for all $n$. Using \eqref{ep}, \eqref{perp1} and \eqref{perp2} we have
\begin{align*}
M &\ge \Vert\sum_{j=1}^n\alpha_{n,j}\phi_{n,j}\Vert^2\\
&=\sum_{j=1}^n\vert\alpha_{n,j}\vert^2\Vert\phi_{n,j}\Vert^2 + \sum_{i\ne j}^d\alpha_{n,j}\overline{\alpha}_{n,i}\langle\phi_{n,j},\phi_{n,i}\rangle\\
&\ge\sum_{j=1}^n\vert\alpha_{n,j}\vert^2\Vert\phi_{n,j}\Vert^2 - \sum_{i\ne j}^d\vert\alpha_{n,j}\overline{\alpha}_{n,i}\vert\varepsilon_n\\
&\ge\sum_{j=1}^n\vert\alpha_{n,j}\vert^2\Vert\phi_{n,j}\Vert^2 - (d^2-d)\max_{1\le j\le d}\big\{\vert\alpha_{n,j}\vert^2\big\}\varepsilon_n\\
&\ge \bigg(\min_{1\le j\le n}\big\{\Vert\phi_{n,j}\Vert^2\big\}- (d^2-d)\varepsilon_n\bigg)\max_{1\le j\le d}\big\{\vert\alpha_{n,j}\vert^2\big\}.
\end{align*}
From which it follows that for some $K\in\mathbb{R}$ we have $\max_{1\le j\le d}\{\vert\alpha_{n,j}\vert\}\le K$ for all $n$.
Consider a sequence $(t_n)\subset[0,1]$ and the vectors $\psi_n(t_n)\in\cL_n(t_n)$ given by
\begin{align*}
\psi_n(t_n) &= \sum_{j=1}^d\alpha_{n,j}\phi_{n,j}(t_n) + \sum_{j=d+1}^n\alpha_{n,j}\phi_{n,j}\\
&=\sum_{j=1}^n\alpha_{n,j}\phi_{n,j} + (I-\hat{P}_n)\sum_{j=1}^d\alpha_{n,j}t_n\phi_j\\
&=\psi_n + (I-\hat{P}_n)\sum_{j=1}^d\alpha_{n,j}t_n\phi_j.
\end{align*}
We have
\[
\Vert\psi_n(t_n)-\psi_n\Vert_{\frak{t}} \le \sum_{j=1}^d t_n\vert\alpha_{n,j}\vert\Vert(I-\hat{P}_n)\phi_j\Vert_{\frak{t}}\le K\sum_{j=1}^d\Vert(I-\hat{P}_n)\phi_j\Vert_{\frak{t}}\to 0,
\]
and therefore the sequence $(\cL_n(t_n))$ is dense in $\mathcal{H}_{\frak{t}}$.

Let $S_n(t)$ be the operator acting on $\cL_n(t)$ which is associated to the restriction
of the form $\frak{s}$ to $\cL_n(t)$.
We now show that for all sufficiently large $n\in\mathbb{N}$ we have
\begin{equation}\label{con1}
\Gamma\subset\rho(S_n(t))\quad\textrm{for all}\quad t\in[0,1].
\end{equation}
We suppose that \eqref{con1} is false. Then there exist sequences $(w_{j})\subset\Gamma$ and $(t_{j})\subset[0,1]$, and a subsequence $n_j$, such that
\begin{displaymath}
\min_{\phi\in\cL_{n_j}(t_{j})\backslash\{0\}}\left\{\frac{\Vert(S_{n_j}(t_{j}) - w_{j})\phi\Vert}{\Vert\phi\Vert}\right\}= 0.
\end{displaymath}
However, the sequence of subspaces $(\cL_{n_j}(t_{j}))$ is dense in $\mathcal{H}_{\frak{t}}$, then by Theorem \ref{resolvent_bound} the set $\Gamma$ is an
$S_{n_j}(t_j)$-regular set. The assertion \eqref{con1} follows from the contradiction. Consider the projection
\begin{displaymath}
P_n(\Gamma,t):=-\frac{1}{2i\pi}\int_{\Gamma}(S_n(t)-\zeta)^{-1}~d\zeta:\cL_n(t)\to\cL_n(t)
\end{displaymath}
which is the spectral projection associated to $S_n(t)$ and the part of the spectrum enclosed by the circle $\Gamma$.
By employing the Gram-Schmidt procedure we may obtain
\begin{displaymath}
\cL_n(t)=\Span\{\hat{\phi}_{n,1}(t),\dots,\hat{\phi}_{n,k}(t),\hat{\phi}_{n,k+1}(t),\dots,\hat{\phi}_{n,n}(t)\}
\end{displaymath}
where the vectors $\{\hat{\phi}_{n,1}(t),\dots,\hat{\phi}_{n,k}(t),\hat{\phi}_{n,k+1}(t),\dots,\hat{\phi}_{n,n}(t)\}$ are analytic in $t$ and orthonormal
for each fixed $t\in[0,1]$. Let $\mathcal{S}_n(t)$ be the matrix representation
of $S_n(t)$ with respect to this orthonormal basis, then
$\mathcal{S}_n(t):\mathbb{C}^n\to\mathbb{C}^n$ clearly has the same eigenvalues as $S_n(t)$, and the eigenvalues have the same multiplicities. Evidently, the
spectral projection associated to $\mathcal{S}_n(t)$ and the part of the spectrum enclosed by the circle $\Gamma$ is analytic in $t$, and therefore
has constant rank for $t\in[0,1]$. We deduce that $\rank(P_n(\Gamma,t))$ is also is constant for $t\in[0,1]$. The result now follows from Lemma \ref{mult}.
\end{proof}

\section{Approximation of $\sigma_{\dis}(T)$}

We consider now the perturbation $T+iQ$ for an orthogonal projection $Q$.
If $(T+iQ - z)\psi=0$ for some $\psi\ne 0$ and $z\notin\mathbb{R}$, then we have
\begin{align*}
(T - \Re z)\psi = i\Im z\psi - iQ\psi,\quad\langle T\psi,\psi\rangle = \Re z\Vert\psi\Vert^2,\quad\Vert Q\psi\Vert^2=\Im z\Vert\psi\Vert^2,
\end{align*}
and hence
\begin{align*}
\Vert(T-\Re z)\psi\Vert^2 &= (\Im z)^2\Vert\psi\Vert^2 + (1-2\Im z)\Vert Q\psi\Vert^2\\
&= \Im z(1-\Im z)\Vert\psi\Vert^2
\end{align*}
from which we obtain the estimate
\begin{equation}\label{est1}
\dist(\Re z,\sigma(T))\le\sqrt{\Im z(1-\Im z)}.
\end{equation}
Therefore, information about the location of $\sigma(T)$ may be gleamed by studying the perturbation $T+iQ$ for a suitably chosen projection. In fact,
the estimate \eqref{est1} could be significantly improved if some a priori information is at hand. Let $(a,b)\cap\sigma(T)=\lambda$,
then using \cite[Lemma 1 \& 2]{kat} we obtain
\begin{equation}\label{int3}
\lambda\in\left(\Re z - \frac{\Im z(1-\Im z)}{b-\Re z},\Re z + \frac{\Im z(1-\Im z)}{\Re z - a}\right)
\end{equation}
whenever the interval on the right hand side is contained in $(a,b)$.

Let $a,b\in\rho(T)$ with $a<b$ and set $\Delta = [a,b]$. For the remainder of this section we assume that $\Delta\cap\sigma(T)=\{\lambda_1,\dots,\lambda_d\}\subset\sigma_{\dis}(T)$ where the eigenvalues are repeated according to multiplicity. The corresponding spectral projection and eigenspace are denoted by $E(\Delta)$ and $\cL(\Delta)$, respectively. Denote by $\Gamma_a$ and $\Gamma_b$, circles with radius $1$ and centers $a$ and $b$, respectively. We define
\begin{align*}U(a,b) := \Big\{z\in\mathbb{C}:&~a<\Re z<b,~z\textrm{ belongs to the exterior of circles }\Gamma_a\textrm{ and }\Gamma_b,\\&~z\ne \lambda_j+i\textrm{ for }1\le j\le d\Big\},
\end{align*}
and, for a compact set $X\subset U(a,b)$
\begin{displaymath}
d_X := \dist(X,\sigma(T)\backslash\Delta)\quad\textrm{and}\quad
d_\Delta:=\dist(\{\lambda_1+i\dots,\lambda_d+i\},X).
\end{displaymath}

\begin{lemma}\label{f}
Let $\varepsilon=\Vert (I-Q)E(\Delta)\Vert$ and $c=\min\{d_X-1-2\varepsilon,d_\Delta-3\varepsilon\}$,
then
\begin{equation}\label{resolvent_bound2}
\Vert (T+iQ -z)u\Vert\ge c\Vert u\Vert\quad\textrm{for all}\quad z\in X\quad\textrm{and}\quad u\in\Dom(T).
\end{equation}
\end{lemma}
\begin{proof}
Let $z\in X$, $u\in\Dom(T)$ and $E:=E(\Delta)$. We have $\Vert(I-E)QE\Vert\le\varepsilon$ and therefore
$\Vert EQ(I-E)\Vert\le\varepsilon$, then using the equality
$Q = EQE +(I-E)QE + EQ(I-E) + (I-E)Q(I-E)$, we obtain
\begin{align*}
\Vert (T+iQ -z)u\Vert & =\Vert(T-z)(I-E)u + (T-z)Eu + iQu\Vert\\
&= \Vert(T-z)(I-E)u + (T-z)Eu\\
&\qquad + i(EQE +(I-E)QE + EQ(I-E) + (I-E)Q(I-E))u\Vert\\
&\ge \Vert(T-z)(I-E)u + i(I-E)Q(I-E)u\\
&\qquad +(T-z)Eu+ iEQEu\Vert - \Vert(I-E)QE + EQ(I-E)\Vert\\
&\ge\Vert(T-z)(I-E)u + i(I-E)Q(I-E)u\\
&\qquad+ (T-z)Eu + iEQEu\Vert - 2\varepsilon\Vert u\Vert.
\end{align*}
The vector $(T-z)(I-E)u$ satisfies the estimate $\Vert(T-z)(I-E)u\Vert\ge d_X\Vert(I-E)u\Vert$,
hence $\Vert(T-z)(I-E)u + i(I-E)Q(I-E)u\Vert\ge(d_X - 1)\Vert(I-E)u\Vert$.
The vector $(T-z)Eu + iEQEu$ satisfies the estimate
\begin{align*}
\Vert(T-z)Eu + iEQEu\Vert &= \Vert(T-z+i)Eu + iE(Q-I)Eu\Vert
\ge (d_\Delta - \varepsilon)\Vert Eu\Vert.
\end{align*}
Combining these estimates yields required result.
\end{proof}

We denote by $U_\varepsilon(a,b)$ the open set contained in $U(a,b)$ and which is exterior to the circles with centers $a$, $b$ and
radius $1+2\varepsilon$ and the circles with center $\lambda_j+i$ and radius $3\varepsilon$ for $1\le j\le d$. An immediate consequence of Lemma \ref{f} is the inclusion
\begin{equation}\label{incl}
U_\varepsilon(a,b)\subset\rho(T+iQ).
\end{equation}

\begin{lemma}\label{cor}
Let $\Vert (I-Q)E(\Delta)\Vert=0$,
then $\lambda_1+i,\dots,\lambda_d+i\in\sigma(T+iQ)$ with spectral subspace $\cL(\Delta)$, and $U(a,b)\subset\rho(T+iQ)$.
\end{lemma}
\begin{proof}
The last assertion is an immediate consequence of \eqref{incl}. Let $(T+iQ - (\lambda_j+i))\phi=0$, then
$(T-\lambda_j)\phi = i(I-Q)\phi$ and therefore
\begin{displaymath}
\langle(T-\lambda_j)\phi,\phi\rangle = i\langle(I-Q)\phi,\phi\rangle.
\end{displaymath}
The left hand side is real and the right hand side is purely imaginary, from which we deduce that $(I-Q)\phi=0$, therefore $(T-\lambda_j)\phi=0$ and hence $\phi\in\cL(\Delta)$.
It follows that $\cL(\Delta)$ is the space spanned by the eigenvectors associated to $T+iQ$ and the eigenvalues $\lambda_1+i,\dots,\lambda_d+i$. Suppose that $\lambda_j+i$ is not
semi-simple. The geometric eigenspace associated to $T + iQ$ and eigenvalue $\lambda_j+i$ is precisely $\cL(\{\lambda_j\})$
the eigenspace associated to $T$ and eigenvalue $\lambda_j$. There exists a non-zero vector $\psi\perp\cL(\{\lambda_j\})$ with
$(T+iQ-\lambda_j-i)\psi = \phi\in\cL(\{\lambda_j\})$.
We have
\begin{displaymath}
(T-(\lambda_j+i))\psi\perp\cL(\{\lambda_j\})\quad\textrm{with}\quad\Vert(T-(\lambda_j+i))\psi\Vert>\Vert\psi\Vert,
\end{displaymath}
and
\begin{displaymath}
iQ\psi = \phi - (T-(\lambda_j+i))\psi\quad\textrm{where}\quad\phi\perp(T-(\lambda_j+i))\psi.
\end{displaymath}
It follows that $\Vert Q\psi\Vert^2 = \Vert\phi\Vert^2 + \Vert(T-(\lambda_j+i))\psi\Vert^2 > \Vert\psi\Vert^2$,
which is a contradiction since $\Vert Q\Vert =1$.
\end{proof}

\begin{theorem}\label{Q}
Let $Q$ be finite rank, $\Vert (I-Q)E(\Delta)\Vert=\varepsilon<1/\sqrt{d}$,
\begin{equation}\label{circles1}
3\varepsilon<r<\min\Big\{\sqrt{(\lambda_j-a)^2+1},\sqrt{(b-\lambda_j)^2+1}\Big\}-1-2\varepsilon\quad\textrm{for}\quad 1\le j\le d
\end{equation}
and $\Gamma_j$ the circle with center $\lambda_j+i$ and radius $r$, and set $X=\cup_{j=1}^d\Gamma_j$.
If $\Gamma_i\cap\Gamma_j=\varnothing$ whenever $i\ne j$, then $\Gamma_j\subset\rho(T+iQ)$, $\dist(\lambda_j+i,\sigma(T+iQ)) < r$,
\[\Vert (T+iQ -z)u\Vert\ge c\Vert u\Vert\quad\textrm{for all}\quad u\in\Dom(T)\quad\textrm{and}\quad z\in\Gamma_j\]
with $c>0$ as in Lemma \ref{f}, and the dimension of the spectral subspace
associated to $T+iQ$ and the region enclosed by $\Gamma_j$ equals the dimension of $\cL(\{\lambda_j\})$.
\end{theorem}
\begin{proof}
An immediate consequence of the condition \eqref{circles1} is that the circle $\Gamma_j$ does not intersect the circles $\Gamma_a$ and $\Gamma_b$. Furthermore,
\begin{displaymath}
c \ge \min\left\{\min_{1\le j\le d}\Big\{\sqrt{(\lambda_j-a)^2+1},\sqrt{(b-\lambda_j)^2+1}\Big\} - r-1-2\varepsilon, r - 3\varepsilon\right\} > 0,
\end{displaymath}
hence $\Gamma_j\subset\rho(T+iQ)$ follows from Lemma \ref{f}. It now suffices to prove the last assertion.

Let $\phi_1,\dots,\phi_e$ form an orthonormal basis for $\cL(\{\lambda_j\})$. Set \[\hat{\phi}_k=Q\phi_k\quad\textrm{and}\quad\hat{\phi}_k(t)=t\phi_k + (1-t)\hat{\phi}_k\quad\textrm{for}\quad t\in[0,1].\]
It is straightforward to show that the condition $\Vert(I-Q)E\Vert< 1/\sqrt{d}\le 1/\sqrt{e}$
implies that $\{\hat{\phi}_1(t),\dots,\hat{\phi}_e(t)\}$ form a linearly independent for any $t\in[0,1]$. Furthermore,
if we set $n=\rank(Q)$, then similarly to the proof of Theorem \ref{multthm} there exist vectors $\{\phi_{e+1},\dots,\phi_n\}$ such that
\begin{align*}
\range(Q)= \Span\{\hat{\phi}_1,\dots,\hat{\phi}_e,\phi_{e+1},\dots,\phi_n\}
\end{align*}
and $\{\hat{\phi}_1(t),\dots,\hat{\phi}_e(t),\phi_{e+1},\dots,\phi_n\}$ is a linearly independent set for every $t\in[0,1]$.
We define the family of orthogonal projections $Q(t)$ such that
\begin{displaymath}
\range(Q(t)) = \Span\{\hat{\phi}_1(t),\dots,\hat{\phi}_e(t),\phi_{e+1},\dots,\phi_n\}.
\end{displaymath}
Let $\phi\in\cL(\{\lambda_j\})$ with $\Vert\phi\Vert=1$, then $\phi = \alpha_1\phi_1+\cdots+\alpha_e\phi_e$ and $\Vert(I-Q)\phi\Vert\le\varepsilon$, therefore
\begin{align*}
\Vert(I-Q(t))\phi\Vert&\le \Vert\phi - \alpha_1\hat{\phi}_1(t)+\cdots+\alpha_e\hat{\phi}_e(t)\Vert\\
&= (1-t)\Vert\alpha_1(\phi_1 -\hat{\phi}_1)+\cdots+\alpha_e(\phi_e-\hat{\phi}_e)\Vert\\
&= (1-t)\Vert(I-Q)\phi\Vert\\
&\le\varepsilon
\end{align*}
and we deduce that $\Gamma_j\subset\rho(T+iQ(t))$ for all $t\in[0,1]$. If $P(t)$ is the spectral projection associated to the operator $T + iQ(t)$ and
the region enclosed by the circle $\Gamma$, then we have
\begin{align*}
P(t)&=-\frac{1}{2i\pi}\int_{\Gamma}(T+iQ(t)-z)^{-1}~dz.
\end{align*}
Evidently, $P(t)$ is a continuous family of projections, therefore $\rank(P(t))=e$ for all $t\in[0,1]$ follows from Lemma \ref{cor}.
\end{proof}

\subsection{Convergence of $\sigma(T+iQ_n)$}
We now assume that $T$ is a bounded self-adjoint operator. Denote by $E_n$ be the spectral measure associated to $P_nT|_{\cL_n}$ and let $Q_n=E_n(\Delta)P_n$. Evidently, $(Q_n)$ is a sequence of finite rank orthogonal projections.

\begin{lemma}\label{Qn}
$Q_n\stackrel{s}{\longrightarrow}E(\Delta)$ and $\Vert(I- Q_n)E(\Delta)\Vert=\mathcal{O}(\delta(\cL(\Delta),\cL_n))$.
\end{lemma}
\begin{proof}
Let $\phi\in\mathcal{H}$, then $\phi = E(\Delta)\phi + (I-E(\Delta))\phi$. For each $1\le j\le d$ we have $P_nTP_nE(\{\lambda_j\})\phi\to TE(\{\lambda_j\})\phi = \lambda_jE(\{\lambda_j\})\phi$,
then it follows from the spectral theorem that $Q_nE(\Delta)\phi\to E(\Delta)\phi$. For any $\mu\in\Delta\cap\rho(T)$ we have $E_n(\mu)P_n\stackrel{s}{\longrightarrow}E(\mu)$ (\cite[Theorem VIII.1.15]{katopert}), from which we deduce that $Q_n(I-E(\Delta))\phi\to 0$.

For the second assertion let $(T-\lambda_j)\psi = 0$ with $\Vert\psi\Vert=1$ and set $\psi_n=P_n\psi$. Then
\begin{align*}
\Vert(P_nT-\lambda_j)\psi_n\Vert &= \Vert(P_nT-\lambda_j)\psi_n - P_n(T-\lambda_j)\psi\Vert\\
&= \Vert P_nT(\psi_n- \psi)\Vert\\
&\le\Vert P_nT\Vert\Vert(I-P_n)\psi\Vert\\
&\le\Vert T\Vert\dist(\psi,\cL_n)\\
&\le\Vert T\Vert\delta(\cL(\Delta),\cL_n),
\end{align*}
and
\begin{align*}
\Vert(I-E_n(\Delta))\psi_n\Vert^2 &= \int_{\mathbb{R}\backslash\Delta}~d\langle (E_n)_\mu\psi_n,\psi_n\rangle\\
&<\int_{\mathbb{R}\backslash(a,b)}\frac{\vert\mu-\lambda_j\vert^2}{\dist[\lambda_j,\{a,b\}]^2}~d\langle (E_n)_\mu\psi_n,\psi_n\rangle\\
&\le\frac{1}{\dist[\lambda_j,\{a,b\}]^2}\int_{\mathbb{R}}\vert\mu-\lambda_j\vert^2~d\langle (E_n)_\mu\psi_n,\psi_n\rangle\\
&=\frac{\Vert(P_nT-\lambda_j)\psi_n\Vert^2}{\dist(\lambda_j,\{a,b\})^2}\\
&\le\frac{\Vert T\Vert^2\delta(\cL(\Delta),\cL_n)^2}{\dist(\lambda_j,\{a,b\})^2}.
\end{align*}
Therefore
\begin{align*}
\Vert(I- E_n(\Delta)P_n)\psi\Vert &\le \Vert(I-E_n(\Delta))\psi_n\Vert + \Vert(I-E_n(\Delta)P_n)(\psi-\psi_n)\Vert\\
&\le \frac{\Vert T\Vert\delta(\cL(\Delta),\cL_n)}{\dist(\lambda_j,\{a,b\})} + \Vert(I-P_n)\psi\Vert\\
&\le\left(\frac{\Vert T\Vert}{\dist(\lambda_j,\{a,b\})} + 1\right)\delta(\cL(\Delta),\cL_n),
\end{align*}
from which the result follows.
\end{proof}

In particular, we have $T+iQ_n\stackrel{s}{\longrightarrow}T+iE(\Delta)$. Let $\cL_n(j)$ be the spectral subspace associated to those eigenvalues $\mu_{n,1},\dots,\mu_{n,e}$ (repeated according to multiplicity) of $T+iQ_n$ which lie in a neighbourhood of $\lambda_j+i$ (see Theorem \ref{Q}) and $\varepsilon_n:=\Vert(I-Q_n)E(\Delta)\Vert$. Theorem \ref{Q}, Lemma \ref{Qn} and \cite[Theorem 6.6]{chat} together imply the following estimate
\begin{equation}\label{chatrates}
\hat{\delta}(\cL(\{\lambda_j\}),\cL_n(j)) = \mathcal{O}(\delta(\cL(\Delta),\cL_n)).
\end{equation}

\begin{lemma}\label{eigconv}
$\max_{1\le k\le e}\vert\mu_{n,k}-\lambda_j-i\vert=\mathcal{O}(\delta(\cL(\Delta),\cL_n)^2)$.
\end{lemma}
\begin{proof}
We argue similarly to the proof of \cite[Theorem 6.11]{chat}. Let $\psi_1,\dots,\psi_e$ be an orthonormal basis for $\cL(\{\lambda_j\})$, then the restriction of $T+iE(\Delta)$ to $\cL(\{\lambda_j\})$ has the matrix representation
\begin{equation}\label{matrixA}
A_{l,k} = \langle(T+iE(\Delta))\psi_k,\psi_l \rangle=(\lambda_j+i)\delta_{lk}.
\end{equation}
It follows from \eqref{chatrates} that
$E(\{\lambda_j\})|_{\cL_n(j)}:\cL_n(j)\to\cL(\{\lambda_j\})$ is a bijection for all sufficiently large $n$. We set
$\psi_{n,k}:=[E(\{\lambda_j\})|_{\cL_n(j)}]^{-1}\psi_k$. Since
\begin{equation}\label{matrixI}
\langle\psi_{n,k},\psi_l\rangle = \langle\psi_{n,k},E(\{\lambda_j\})\psi_l\rangle =\langle\psi_k,\psi_l\rangle=\delta_{lk},
\end{equation}
the restriction of $T+iE_n(\Delta)$ to $\cL_n(j)$ has the matrix representation
\[B_{l,k} = \langle(T+iE_n(\Delta))\psi_{n,k},\psi_l \rangle,\]
and $\mu_{n,1},\dots,\mu_{n,e}$ are the eigenvalues of the matrix $B$. We have
\begin{align*}
\vert A_{lk} - B_{lk}\vert &= \vert\langle(T+iE(\Delta))\psi_k,\psi_l\rangle - \langle(T+iE_n(\Delta))\psi_{n,k},\psi_l\rangle\vert\\
&=\vert\langle(T+iE(\Delta))E(\Delta)\psi_{n,k},\psi_l\rangle - \langle(T+iE_n(\Delta))\psi_{n,k},\psi_l\rangle\vert\\
&=\vert\langle(iE(\Delta)-iE_n(\Delta))\psi_{n,k},\psi_l\rangle\vert\\
&=\vert\langle(I-E_n(\Delta))\psi_{n,k},\psi_l\rangle\vert\\
&=\vert\langle(I-E_n(\Delta))\psi_{n,k},(I-E_n(\Delta))P_n\psi_l\rangle\vert\\
&\le\Vert(I-E_n(\Delta))\psi_{n,k}\Vert\Vert(I-E_n(\Delta))P_n\psi_l\Vert.
\end{align*}
Using Lemma \ref{Qn}, the second term on the right hand side satisfies
\[
\Vert(I-E_n(\Delta))P_n\psi_l\Vert \le \Vert(I-Q_n)\psi_l\Vert + \Vert(I-P_n)\psi_l\Vert=\mathcal{O}(\delta(\cL(\Delta),\cL_n)).
\]
Since $\psi_{n,k}=[E(\{\lambda_j\})|_{\cL_n(j)}]^{-1}\psi_k=:u_n\in\cL_n(j)$ and $\Vert[E(\{\lambda_j\})|_{\cL_n(j)}]^{-1}\Vert\le M$ for some $M>0$ and all sufficiently large $n$, it follows from \eqref{chatrates} that $u_n = v_n + w_n$ where $v_n\in\cL(\Delta)$ and $\Vert w_n\Vert\le \mathcal{O}(\delta(\cL(\Delta),\cL_n))$. Hence
\begin{align*}
\Vert(I-E_n(\Delta))\psi_{n,k}\Vert & = \Vert(I-E_n(\Delta))u_n\Vert\\
&\le\Vert(I-E_n(\Delta))P_nv_n\Vert + \mathcal{O}(\delta(\cL(\Delta),\cL_n))\\
&=\mathcal{O}(\delta(\cL(\Delta),\cL_n)).
\end{align*}
Combining these estimates we obtain $\Vert A - B\Vert_{\mathbb{C}^e} = \mathcal{O}(\delta(\cL(\Delta),\cL_n)^2)$. The result follows from this estimate, \eqref{matrixA} and \eqref{matrixI}.
\end{proof}

We denote by $U_{\tau,r}(a,b)$ the compact set enclosed by the rectangle $\{z\in\mathbb{C}:a\le\Re z\le b\textrm{ and }0\le\Im z \le 1\}$ and exterior to the circles with centers $a$, $b$ and radius $1+\tau$ and the circles with center $\lambda_j+i$ and radius $r^2$ for $1\le j\le d$. We have proved the following Theorem.

\begin{theorem}\label{QQ}
There exist sequences $(\tau_n)$ and $(r_n)$ of non-negative reals, with
\[0\le \tau_n = \mathcal{O}(\delta(\cL(\Delta),\cL_n))\quad\textrm{and}\quad 0\le r_n = \mathcal{O}(\delta(\cL(\Delta),\cL_n)^2),\]
such that $U_{\tau_n,r_n}(a,b)\backslash\mathbb{R}\subset\rho(T+iQ_n)$ for all sufficiently large $n$.
Moreover, if $\Gamma_j$ is the circle with center $\lambda_j+i$ and radius $r_n$, then $\Gamma_j\subset\rho(T+iQ_n)$ and the dimension of the spectral subspace associated to $T+iQ_n$ and the region enclosed by $\Gamma_j$ equals the dimension of $\cL(\{\lambda_j\})$.
\end{theorem}

\subsection{Convergence of $\sigma(T+iQ_{m},\cL_n)$}
In this section we assume that $T$ is bounded and $m\in\mathbb{N}$ is fixed. We consider the Galerkin approximation of a non-real eigenvalue $\mu\in\sigma(T+iQ_m)$ which lies in a neighbourhood of $\lambda_j+i$. First, we note that by Theorem \ref{spectralinc} there is no spectral pollution away from the real line, hence
\[\sigma(T+iQ_m,\cL_\infty)\backslash\mathbb{R} = \sigma(T+iQ_m)\backslash\mathbb{R}.\]
If $\Gamma\subset\rho(T+iQ_m)$ is a circle with center $\mu$, which encloses no other element from $\sigma(T+iQ_m)$ and does not intersect $\mathbb{R}$, then by Theorem \ref{multthm}, for all sufficiently large $n$ the multiplicity of those elements from $\sigma(T+iQ_m,\cL_n)$ enclosed by $\Gamma$ is equal to the multiplicity of $\mu$. Furthermore, by Theorem \ref{resolvent_bound}, $\Gamma$ is a $P_n(T+iQ_m)|_{\cL_n}$-regular set. With these three properties, the sequence of operators $P_n(T+iQ_m)P_n$ is said to be a strongly stable approximation of $T + iQ_m$ in the interior of $\Gamma$; see \cite[Section 5.2 \& 5.3]{chat}. This allows the application of the following well-known super-convergence result for strongly stable approximations.

Let $\mathcal{M}$ (respectively $\mathcal{M^*}$) be the spectral subspace associated to the operator $T+iQ_m$ (respectively $T-iQ_m$) and eigenvalue $\mu$ (respectively $\overline{\mu}$). Let $\mu$ have algebraic multiplicity $e$ and let $z_1,\dots,z_e$ be those (repeated) eigenvalues from $\sigma(T+iQ_m,\cL_n)$ which lie in a neighbourhood of $\mu$ and set $\hat{z}_n = (z_1+\dots+z_e)/e$, then
\begin{equation}\label{chatsuper}
\vert\hat{z}_n - \mu\vert = \mathcal{O}(\delta(\mathcal{M},\cL_n)\delta(\mathcal{M}^*,\cL_n));
\end{equation}
see \cite[Theorem 6.11]{chat}.

We therefore have the following strategy for approximating the eigenvalues in the region $\Delta$:
\newcounter{counter_assump1}
\begin{list}{{\rm\textbf{(\arabic{counter_assump1})}}}%
{\usecounter{counter_assump1}
\setlength{\itemsep}{0.5ex}\setlength{\topsep}{1.1ex}
\setlength{\leftmargin}{7ex}\setlength{\labelwidth}{7ex}}
\item calculate $\sigma(T,\cL_m)$ and choose $Q_m=E_m(\Delta)P_m$
\item calculate $\sigma(T+iQ_m,\cL_n)$ for $\dim \cL_m\ll\dim\cL_n$.
\end{list}

\begin{example}
With $\mathcal{H}=L^2(-\pi,\pi)$ we consider the bounded self-adjoint operator
\[
T\phi = a(x)\phi + 10\langle\phi,\psi_0\rangle\psi_0\quad\textrm{where}\quad a(x)=\begin{cases}
  -2\pi - x & \text{for}\quad -\pi<x\le 0, \\
  2\pi - x & \text{for}\quad 0<x\le\pi,
  \end{cases}
\]
and $\psi_k = e^{-ikx}$ for $k\in\mathbb{Z}$. We have $\sigma_{\ess}(T)=[-2\pi,-\pi]\cup[\pi,2\pi]$, and $\sigma_{\dis}(T)$ consists of the two simple eigenvalues $\lambda_1\approx -1.64834270$ and $\lambda_2\approx 11.97518502$; see \cite[Lemma 12]{DP}. We note that the eigenvalue $\lambda_1$ lies in the gap in $\sigma_{\ess}(T)$.

\begin{figure}[h!]
\centering
\includegraphics[scale=.45]{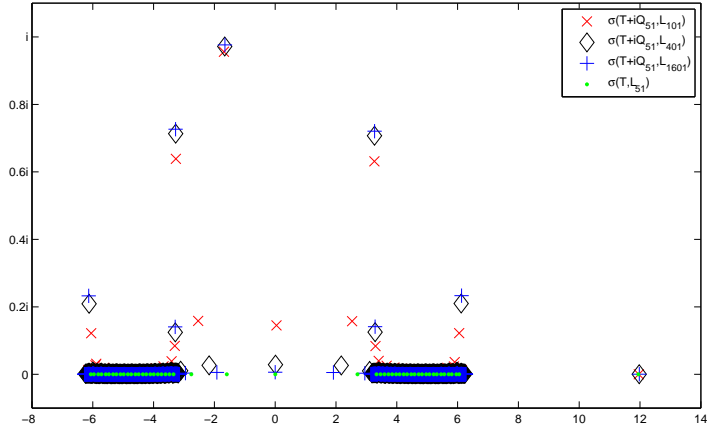}
\caption{$\sigma(T+iQ_{51},L_n)$ for $n=101, 401, 1601$ and $\sigma(T,L_{51})$.}
\end{figure}

Let $\cL_{2n+1} = \Span\{e^{-inx},\dots,e^{inx}\}$. We find that $\sigma(T,\cL_{51})$ has four eigenvalues in the interval $(-\pi,\pi)$. With $Q_{51}=E_{51}((-\pi,\pi))P_{51}$ we calculate $\sigma(T+iQ_{51},\cL_{2n+1})$ for $n=50,200$ and $800$. The results are displayed in Figure 1, and, consistent with Theorem \ref{QQ}, suggest that $\sigma(T+iQ_{51})$ has a simple eigenvalue near $\lambda_1 + i$.
\begin{figure}[h!]
\centering
\includegraphics[scale=.45]{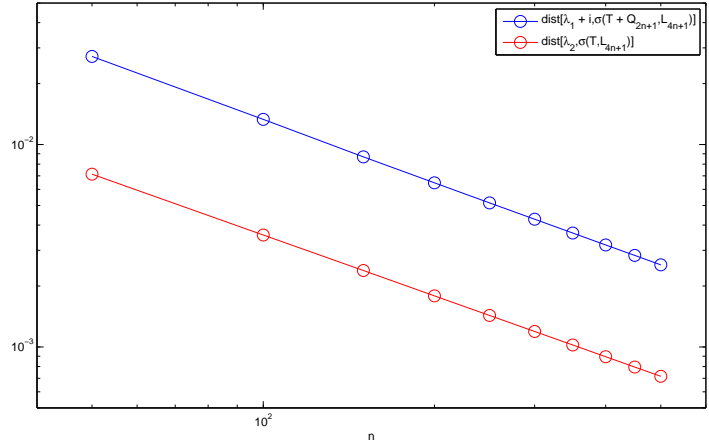}
\caption{Convergence to $\lambda_1+i$ using $\sigma(T+iQ_{2n+1},L_{4n+1})$ compared to the convergence to $\lambda_2$ using $\sigma(T,L_{4n+1})$.}
\end{figure}
Calculating $\sigma(T+iQ_{2n+1},\cL_{20n+1})$ with $n=4,8,12,\dots,40$ suggests that we have
\begin{equation}\label{seq1}
\dist(\lambda_1+i,\sigma(T+iQ_{2n+1},\cL_{20n+1})) \approx \mathcal{O}(n^{-1}).
\end{equation}
The following estimate holds:
\begin{equation}\label{perprates2}
\delta(\cL(\{\lambda_1\}),\cL_{2n+1})=\mathcal{O}(n^{-1/2});
\end{equation}
see for example \cite[Lemma 3.1]{bo}. Combining this estimate with Theorem \ref{QQ} we obtain
\begin{equation}\label{perprates0}
\dist(\lambda_1+i,\sigma(T+iQ_n)) = \mathcal{O}(n^{-1})
\end{equation}
which is consistent with \eqref{seq1}. The latter suggests that in Theorem \ref{QQ} the convergence rate for $r_n$ is sharp.

For a fixed and sufficiently large $m$ we denote by $\mathcal{M}$ (respectively $\mathcal{M}^*$) the eigenspace associated to the simple eigenvalue of $T+iQ_m$ (respectively $T-iQ_m$) which lies in a neighbourhood of $\lambda_1+i$ (respectively $\lambda_1-i$). Then the following estimates hold:
\begin{equation}\label{perprates}
\delta(\mathcal{M},\cL_n)=\mathcal{O}(n^{-1/2})\quad\textrm{and}\quad\delta(\mathcal{M}^*,\cL_n)=\mathcal{O}(n^{-1/2});
\end{equation}
see for example \cite[Lemma 3]{bo}.
For the approximation of the eigenvalue $\lambda_j$ we calculate $\sigma(T+iQ_{2n+1},\cL_{4n+1})$ with $n=50,100,150,\dots,500$. For comparison, we also approximate the eigenvalue $\lambda_2$ which lies outside the convex hull of the essential spectrum and may therefore be approximated without encountering spectral pollution. The results are displayed in Figure 2 and suggest that
\begin{align}
&\dist(\lambda_1+i,\sigma(T+iQ_{2n+1},\cL_{4n+1}))=\mathcal{O}(n^{-1})\quad\textrm{and}\label{rate1}\\
&\dist(\lambda_2,\sigma(T,\cL_{4n+1}))=\mathcal{O}(n^{-1})\label{rate2}.
\end{align}
The convergence in \eqref{rate1} is consistent with \eqref{perprates0}, \eqref{perprates} and \eqref{chatsuper}. The convergence
in \eqref{rate2} follows from \eqref{perprates2} and the well-known superconvergence result for an eigenvalue lying outside the convex hull of the essential spectrum of bounded self-adjoint operator.
\end{example}

\subsection{Unbounded Operators}
We now assume that $T$ is bounded from below and unbounded from above. Let $\gamma<\min\sigma(T)$ and consider the operator $T-\gamma$. We have
$[a-\gamma,b-\gamma]\cap\sigma_{\ess}(T-\gamma)=\varnothing$, and, in particular
\begin{equation}\label{res}
\left[\frac{1}{b-\gamma},\frac{1}{a-\gamma}\right]\cap\sigma_{\ess}((T-\gamma)^{-1})=\varnothing.
\end{equation}
We shall approximate the eigenvalues $\sigma(T)\cap[a,b]$ by applying the results from the preceding sections to approximate eigenvalues of $(T-\gamma)^{-1}$ in $\sigma((T-\gamma)^{-1})\cap[(b-\gamma)^{-1},(a-\gamma)^{-1}]$.

Let $\{\mu_1,\dots,\mu_k\} = [a,b]\cap\sigma(T,\cL_n)$
where the eigenvalues are repeated according to multiplicity.
Let $\{u_1,\dots,u_k\}$ be a corresponding set of orthonormal eigenvectors.
We write $\hat{x} = (T-\gamma)^{\frac{1}{2}}x$ and
$\hat{\cL}_n = (T-\gamma)^{\frac{1}{2}}\cL_n$, and note that $\hat{u}_i\perp\hat{u}_j$ for $i\ne j$ since
\[\langle\hat{u}_i,\hat{u}_j\rangle = (\frak{t}-\gamma)(u_i,u_j) = (\mu_i-\gamma)\langle u_i,u_j\rangle = 0.\]
Furthermore, we have for some $\mu_i\in[a,b]$ and any $y\in\cL_n$
\[
0 = (\frak{t}-\gamma)(u_i,y) - (\mu_i-\gamma)\langle u_i,y\rangle = \langle\hat{u}_i,\hat{y}\rangle - (\mu_i-\gamma)\langle(T-\gamma)^{-1}\hat{u}_i,\hat{y}\rangle,
\]
so that $(\mu_i-\gamma)^{-1}\in\sigma((T-\gamma)^{-1},\hat{\cL}_n)$. Evidently, there is a one-to-one correspondence between $\sigma(T,\cL_n)$
and $\sigma((T-\gamma)^{-1},\hat{\cL}_n)$:
\[
\sigma((T-\gamma)^{-1},\hat{\cL}_n) = \left\{\frac{1}{\lambda-\gamma}:~\lambda\in\sigma(T,\cL_n)\right\}.
\]
In particular, we have
\[
\sigma((T-\gamma)^{-1},\hat{\cL}_n)\cap\left[\frac{1}{b-\gamma},\frac{1}{a-\gamma}\right] = \left\{\frac{1}{\mu_1-\gamma},\dots,\frac{1}{\mu_k-\gamma}\right\}
\]
with corresponding orthogonal eigenvectors given by $\{\hat{u}_1,\dots,\hat{u}_k\}$. Denote by $Q_n$ the orthogonal projection onto $\spn\{\hat{u}_1,\dots,\hat{u}_k\}\subset\hat{\cL}_n$. From the first paragraph in the proof of lemma \ref{subspacegap} it follows that $\delta(\cL(\Delta),\hat{\cL}_n)=\mathcal{O}(\delta_{\frak{t}}(\cL(\Delta),\cL_n))$, then by Lemma \ref{Qn}
we have
 \[Q_n\stackrel{s}{\longrightarrow}E(\Delta)\quad\textrm{and}\quad\Vert(I- Q_n)E\Vert=\mathcal{O}(\delta_{\frak{t}}(\cL(\Delta),\cL_n)).\]
Hence, a direct application of Terrorem \ref{QQ} to the bounded self-adjoint operator $(T-\gamma)^{-1}$ and subspaces $(\hat{\cL}_n)$ yields the following corollary.

\begin{corollary}\label{QQQ}
There exists sequences $(\tau_n)$ and $(r_n)$ of non-negative reals, with
\[0\le \tau_n = \mathcal{O}(\delta_{\frak{t}}(\cL(\Delta),\cL_n))\quad\textrm{and}\quad 0\le r_n = \mathcal{O}(\delta_{\frak{t}}(\cL(\Delta),\cL_n)^2),\]
such that $U_{\tau_n,r_n}(1/(b-\gamma),1/(a-\gamma))\backslash\mathbb{R}\subset\rho((T-\gamma)^{-1}+iQ_n)$ for all sufficiently large $n$.
Moreover, if $\Gamma_j$ is the circle with center $1/(\lambda_j-\gamma)+i$ and radius $r_n$, then $\Gamma_j\subset\rho((T-\gamma)^{-1}+iQ_n)$ and the dimension of the spectral subspace associated to $(T-\gamma)^{-1}+iQ_n$ and the region enclosed by $\Gamma_j$ equals the dimension of $\cL(\{\lambda_j\})$.
\end{corollary}

For a fixed $m$, let $u_1,\dots,u_k$ be as above, and consider the eigenvalue problem: find $z\in\mathbb{C}$ for which there exists an $x\in\cL_n\backslash\{0\}$ with
\begin{equation}\label{in1}
(\frak{t}-\gamma)(x,y)- iz\sum_{j=1}^k\frac{(\frak{t}-\gamma)(x,u_j)(\frak{t}-\gamma)(u_j,y)}{(\frak{t}-\gamma)[u_j]} - z\langle x,y\rangle = 0\quad\forall~y\in\cL_n.
\end{equation}
Evidently, this is equivalent to the eigenvalue problem: find $z\in\mathbb{C}$ for which there exists a $\hat{x}\in\hat{\cL}_n\backslash\{0\}$ and
\begin{equation}\label{in2}
\langle\hat{x},\hat{y}\rangle - iz\sum_{j=1}^k\frac{\langle\hat{x},\hat{u}_j\rangle\langle\hat{u}_j,\hat{y}\rangle}{\Vert\hat{u}_j\Vert^2} - z\langle(T-\gamma)^{-1}\hat{x},\hat{y}\rangle = 0\quad\forall~\hat{y}\in\hat{\cL}_n.
\end{equation}
The solutions to \eqref{in2} are precisely the
set $\{w^{-1}:~w\in\sigma((T-\gamma)^{-1}+ iQ_m,\hat{\cL}_n)\}$. Therefore, we may approximate the eigenvalues $\sigma((T-\gamma)^{-1})\cap[(b-\gamma)^{-1},(a-\gamma)^{-1}]$ by solving \eqref{in1}.

\begin{example}
With $\mathcal{H}=L^2(0,1)\oplus L^2(0,1)$ we consider the block operator matrix
\[
T = \left(
\begin{array}{cc}
-d^2/dx^2 & -d/dx\\
d/dx & 2I
\end{array} \right)
\]
with homogeneous Dirichlet boundary conditions in the first component. The same matrix (but with different boundary conditions) has been studied in \cite{lesh}. We have $\sigma_{\ess}(T)=\{1\}$ (see for example \cite[Example 2.4.11]{Tretter}) while
$\sigma_{\dis}(T)$ consists of the simple eigenvalue $\{2\}$ with eigenvector $(0,1)^T$, and the two sequences of simple eigenvalues
\begin{displaymath}
\lambda_k^\pm := \frac{2+k^2\pi^2 \pm\sqrt{(k^2\pi^2 + 2)^2  - 4k^2\pi^2}}{2}.
\end{displaymath}
The sequence $\lambda_k^-$ lies below, and accumulates at, the essential spectrum. While the sequence $\lambda_k^+$ lies above the eigenvalue $2$ and accumulates at $\infty$. Therefore, we have
\[\sigma_{\ess}(T^{-1})=\{0,1\}\quad \textrm{and}\quad\sigma_{\dis}(T^{-1})=\bigcup_{k=1}^{\infty}\left\{\frac{1}{\lambda_k^+}\right\}\bigcup_{k=1}^{\infty}\left\{\frac{1}{\lambda_k^-}\right\}\bigcup\left\{\frac{1}{2}\right\}.\]
\begin{figure}[h!]
\centering
\includegraphics[scale=.45]{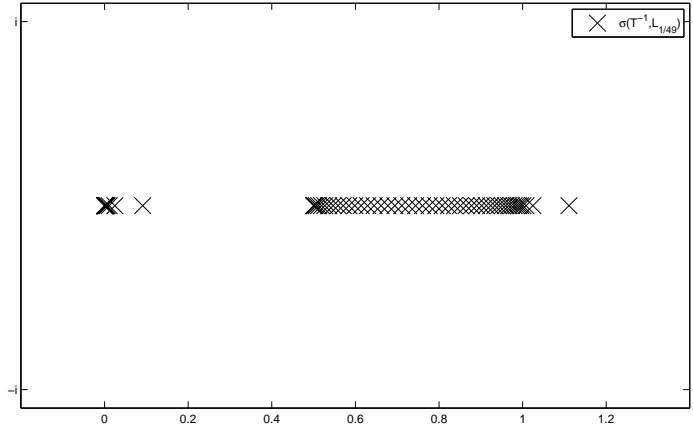}
\caption{$\sigma(T^{-1},L_{1/49})$ displaying spectral pollution in the interval $(1/2,1)$.}
\end{figure}
\begin{figure}[h!]
\centering
\includegraphics[scale=.45]{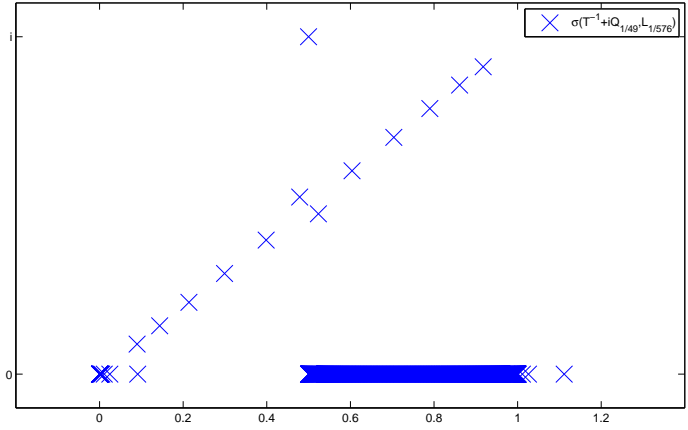}
\caption{$\sigma(T^{-1}+iQ_{1/49},L_{1/576})$ displaying the approximation of the eigenvalue $1/2+i$.}
\end{figure}
Denote by $\cL_h^0$ the FEM space of piecewise linear functions on a uniform mesh of size $h$ and satisfying homogeneous Dirichlet boundary conditions, and by $\cL_h$ the space without boundary conditions. The subspaces $\cL_h^0\oplus \cL_h$ belong to $\Dom(\frak{t})$. We define $L_{h}=T^{\frac{1}{2}}(\cL_h^0\oplus \cL_h)$.

Figure 3 shows $\sigma(T^{-1},L_{1/49})$. The interval $(1/2,1)$ is filled with Galerkin eigenvalues, however, the interval $(1/2,1)$ belongs to the resolvent set of $T^{-1}$. This is an example of spectral pollution, the interval lies in the gap in the essential spectrum which is where the Galerkin method is known to be unreliable. We note that the eigenvalue $1/2$ is obscured by the spectral pollution.

Figure 4 shows $\sigma(T^{-1}+iQ_{1/49},L_{1/576})$ where $Q_{1/49}$ is the orthogonal projection associated to $\sigma(T^{-1},L_{1/49})$ and the interval $[1/4,9/10]$. Since $\sigma(T^{-1})\cap[1/4,9/10]$ consists only of the simple eigenvalue $1/2$, the set $\sigma(T^{-1}+iQ_{1/49},L_{1/576})$ has only one element with imaginary part near $1$, in fact, $1/2+i\in\sigma(T^{-1}+iQ_{1/49},L_{1/576})$ because the eigenvector associated to this eigenvalue is $\psi=(0,1)^T\in\cL_h^0\oplus \cL_h$, hence the eigenvalue also belongs to $L_h$. Therefore, our method has identified this eigenvalue, and furthermore, Figure 4 suggests that all elements $\sigma(T^{-1},L_{1/49})\cap(1/2,1)$ are points of spectral pollution.

We now turn to the approximation of the eigenvalue $1/\lambda_1^+$ which lies in the gap in the essential spectrum $(0,1)$.
Figure 5 shows $\sigma(T^{-1}+iQ_{1/49},L_{1/576})$ where $Q_{1/49}$ is now the orthogonal projection associated to $\sigma(T^{-1},L_{1/49})$ and the interval $[1/20,1/5]$.

Since $\sigma(T^{-1})\cap[1/20,1/5]$ consists only of the simple eigenvalue $\{1/\lambda_1^+\}$, the set $\sigma(T^{-1}+iQ_{1/49},L_{1/576})$ has only one element with imaginary part near $1$, this is an approximation of $1/\lambda_1^++i$. The Galerkin method does not appear to suffer from spectral pollution in the interval $(0,1/2)$. Table 1 shows the approximation of $1/\lambda_1^+$ using $\sigma(T^{-1},L_{h/2})$ and $\sigma(T^{-1}+iQ_{h},L_{h/2})$. Both converge to
$1/\lambda_1^+$ with order $\mathcal{O}(h^2)$.

\begin{figure}[h!]
\centering
\includegraphics[scale=.45]{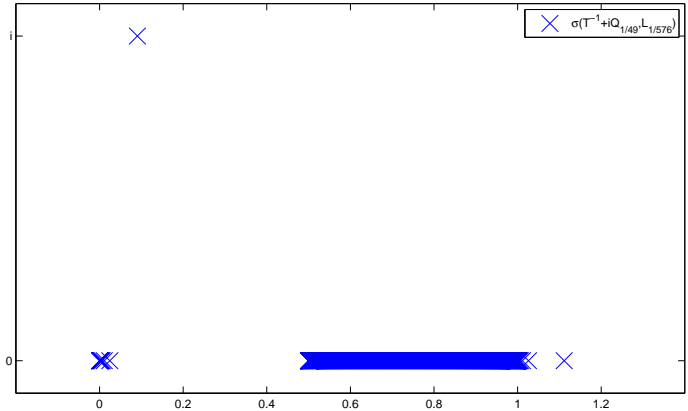}
\caption{$\sigma(T^{-1}+iQ_{1/49},L_{1/576})$ displaying the approximation of the eigenvalue $1/\lambda_1^++i$.}
\end{figure}

\begin{center}
\begin{table}
\begin{tabular}{|c|c|c|}
\hline
h & dist($(\lambda_1^+)^{-1},\sigma(T^{-1},L_{h/2})$) & dist($(\lambda_1^+)^{-1}+i,\sigma(T^{-1}+iQ_{h},L_{h/2})$)\\
\hline
1/9    & 1.852226448408184e-004 & 7.356900130780202e-004\\
1/19   & 4.159849994125886e-005 & 1.656338892411880e-004\\
1/39   & 9.875177553464454e-006 & 3.934129644715903e-005\\
1/79   & 2.406805040600091e-006 & 9.589568304944231e-006\\
1/159  & 5.941634519252004e-007 & 2.367430965052093e-006\\
1/319  & 1.476118197535348e-007 & 5.882392223781511e-007\\
\hline
\end{tabular}
\newline
\newline
\caption{A comparison of the approximation of $(\lambda_1^+)^{-1}$ using $\sigma(T^{-1},L_{h/2})$ and $\sigma(T^{-1}+iQ_{h},L_{h/2})$}
\end{table}
\end{center}

\end{example}

\section{Acknowledgements}
The author is grateful to Marco Marletta for useful discussions and acknowledges the support of the Wales Institute of Mathematical and Computational Sciences and the Leverhulme Trust grant: RPG-167.

\end{document}